\documentclass[12pt,a4paper]{article}
\usepackage[cp1251]{inputenc}
\usepackage[russian]{babel}
\usepackage{indentfirst}
\usepackage{amsmath}
\usepackage{graphicx}
\usepackage{float}
\usepackage{stmaryrd}
\usepackage{bbm}
\usepackage[all]{xy}
\usepackage{amsfonts}
\usepackage{amssymb}

\usepackage{amscd}
\usepackage{mathrsfs}
\usepackage{tipa}
\usepackage{mathtext,amsmath, amstext, amsgen, amsthm}
\usepackage{amsopn, amsfonts, euscript,amscd,amssymb,latexsym}
 \newtheorem{theorem}{Theorem}
 \newtheorem{cor}[theorem]{Corollary}
 \newtheorem{lem}[theorem]{Lemma}
 \newtheorem{prop}[theorem]{Proposition}
 \theoremstyle{definition}
 \newtheorem{defn}[theorem]{Definition}
 \theoremstyle{remark}
 \newtheorem{rem}[theorem]{Remark}
 \newtheorem{ex}{Example}
 
 \numberwithin{equation}{section}

\begin{document}

\begin{center}
{\bf \large BASIC AUTOMORPHISM GROUPS OF COMPLETE \\

\vspace{2mm}
CARTAN FOLIATIONS }\\
\end{center}

\begin{center}
{ N. I. ZHUKOVA~$^\dag$ AND  K. I. SHEINA~$^\ddag$}\\
\end{center}

{\it $^\dag$, $^\ddag$ National Research University Higher School of Economics, \\
Department of Informatics, Mathematics and Computer Science,\\
ul. Bolshaja Pecherskaja, 25/12, Nizhny Novgorod, 603155, Russia\\
} 
\,\,\,E-mail:{ n.i.zhukova@rambler.ru;\,\,\,kse51091@mail.ru}

\vspace{4mm}
We get sufficient conditions for the full basic automorphism group
of a comp\-le\-te Cartan foliation to admit a unique
(finite-dimen\-sio\-nal) Lie group structure in the ca\-te\-go\-ry of
Cartan foliations. In particular, we obtain sufficient conditions
for this group to be discrete. Emphasize that the transverse Cartan
geometry may
  be noneffective. Some estimates of the dimension of this group
depending on the transverse geometry are found. Further, we
investigate  Cartan foliations covered by fibrations and ascertain their
specification. Examples of computing the full ba\-sic \-
auto\-mor\-phism group of complete Cartan foliations are constructed.

{\it \bf Keynwords: }{foliation; Cartan foliation; Lie group; basic automorphism; auto\-mor\-phism group, foliated bundle.}


\section{Introduction. Main results}

 The automorphism group is associated with every object of a category.
Among central problems there is the question whether the automorphism group can
be endowed with a (finite-dimensional) Lie group structure~\cite{K}.

In the theory of foliations with transverse geometries,
morphisms are understood as local diffeomorphisms mapping leaves
onto leaves and preserving transverse geometries. The group of all
automorphisms of a foliation $(M,F)$ with transverse geometry is
denoted by ${A}(M,F).$ Let ${ A}_L(M,F)$ be the normal
subgroup of ${A}(M,F)$ formed by automorphisms mapping each
leaf onto itself. The quotient group ${ A}(M,F)/{A}_L(M,F)$ is called
{\it the full basic automorphism group} and is denoted by ${A}_B(M,F).$

In the investigation of foliations $(M,F)$ with transverse
geometry it is natural to put the above problem of the existence of a Lie
group structure for the full group ${A}_B(M,F)$ of basic automorphisms of $(M,F).$

J. Leslie \cite{Les} was the first who  solved a similar problem for smooth
foliations on compact manifolds. For foliations with complete
transversally projectable affine connec\-tion this problem was raised
by I.V. Belko~\cite{Bel}.

Foliations $(M,F)$ with effective transverse rigid geometries were
investigated by the first author \cite{ZhR} where an algebraic
invariant $\mathfrak g_0 = \mathfrak g_0(M, F),$ called the structural Lie
algebra of $(M,F)$, was constructed and it was proved that $\mathfrak g_0 = 0$
is a sufficient condition for the  existence of a unique  Lie
group structure in the full basic automorphism group of this
foliation. In the case where $(M, F)$ is a Rie\-man\-nian foliation, the
concept of the structural Lie algebra was introduced previously by
P. Molino \cite{Mo}.

Spaces which we call Cartan geometries were introduced by Elie Cartan in the 1920s
and were called by him {\it espaces g$\acute{e}$n$\acute{e}$raliz$\acute{e}$d}.
The investigation of Cartan geometries (see definition \ref{d3}) gives the
possibility to consider different geometry structures from the unified viewpoint.

We use the notion of Cartan foliation in the sense of R. Blumenthal \cite{Bl}.
We emphasize that the following classes of foliations:  parabolic, conformal,
 Weil,  projective,  pseudo-Riemannian, Lorentzian, Rie\-man\-nian foliations
and foliations with transverse linear connection  belong to  Cartan
foliations. Therefore, all proved by us theorems and corollaries are
valid for all these foliations. Let us denote by $\mathfrak
C\mathfrak F$ the category of Cartan foliations (the definition is
given in subsection \ref{ss2.2}).

In subsection \ref{sseff} we remind the notion of the effective Cartan
geometry. It was shown by the first author (\cite{Min}, Proposition
1) that a Cartan foliation modelled on a noneffective  Cartan
geometry $\xi=(P(N,{H}),\omega)$ of  type $({G},{H})$ admits also an
effective transversal Cartan geometry of the type $({G'},{H'})$
where $G'=G/K$, $H'= H/K$ and $K$ is the kernel of the pair $(G,H)$,
that is the maximal normal subgroup of $G$ belonging to $H$. Due to
this fact we may construct the associated foliated bundle for any
Cartan foliation. Note that in (\cite{Bl},
Proposition 3.1) this construction is not correct in general.

By the structural Lie algebra $\mathfrak g_0 = \mathfrak g_0(M, F)$ of a
complete Cartan foliation  $(M, F)$ we mean the structural Lie
algebra of $(M, F)$ considered with the associated effective
transversal Cartan geometry indicated above.

Let us denote by $A(M, F)$ the  group of all the automorphisms of the
Cartan foliation $(M, F)$ in the category $\mathfrak{C}\mathfrak{F}$ and by
 $A_B(M, F)$ the full basic automorphism group.

Recall that a leaf $L$ of a foliation $(M, F)$ is proper if $L$ is
an embedded submanifold in $M$. A foliation is called proper \cite{Tam} if all
its leaves are proper. A leaf $L$ is said to be closed if $L$ is a
closed subset of $M$. As it is known, any closed leaf is proper.

We get the following theorem about a sufficient condition for the existence a
unique Lie group structure in the group of basic automorphisms
of complete Cartan foliations and some exact estimates of its dimension.

\begin{theorem}\label{Th1} Let $(M, F)$ be a complete  Cartan foliation
modelled on a Cartan geometry of  type $\mathfrak g/\mathfrak h$. If the
structural Lie algebra $\mathfrak g_0 = \mathfrak g_0(M, F)$ is zero, then
the  basic auto\-mor\-phism group $A_B(M, F)$ of this foliation is a
Lie group whose dimension satisfies the inequality
\begin{gather}\label{oz1}
\dim{A}_B(M,F)\leq \dim (\mathfrak g) - \dim(\mathfrak k),
\end{gather}
where ${\mathfrak k}$ is the kernel of the pair $(\mathfrak g,\mathfrak h)$,
i.e., the maximal ideal of the Lie algebra $\mathfrak g$
belonging to $\mathfrak h$,
and the Lie group structure in  $A_B(M, F)$ is unique.

Moreover,
\begin{itemize} \item[(a)] if there exists an isolated proper leaf
or if the set of proper leaves is countable, then
\begin{gather}\label{oz2}
\dim{ A}_B(M,F)\le\dim (\mathfrak h) - dim(\mathfrak k);
\end{gather}
\item[(b)] if the set of proper leaves is countable and dense, then
 \begin{gather}\label{oz3}
\dim{ A}_B(M,F)=0.
\end{gather}
\end{itemize}
The estimates \eqref{oz1}, \eqref{oz2} are exact and the case of $(b)$ is realized.
\end{theorem}
In other words, if the associated lifted foliation $(\mathcal{R},\mathcal{F})$
is formed by fibres of a locally trivial fibration, then the
basic automorphism group of $(M, F)$ is a Lie group.

Examples \ref{E1} -- \ref{E3} show the exactness of estimates \eqref{oz1} and \eqref{oz2}.
In Example \ref{E5} we construct the foliation with the countable dense set of closed leaves
and show the realization of the case $(b)$ of Theorem \ref{Th1}.

The following assertion contains sufficient conditions in terms of topo\-lo\-gy
of leaves and their holonomy groups  for the basic automorphism group of
a Cartan foliation to be a Lie group.

\begin{cor}\label{c1} Let $(M, F)$ be a complete Cartan
foliation. If at least one of the following conditions
holds:
\begin{itemize} \item[(i)] there exists a proper leaf $L$ with discrete
holonomy group (in the sense of definition~\ref{d7});
 \item[(ii)] there is a closed leaf $L$ with discrete holonomy group;
 \item[(iii)] there exists a proper leaf $L$ with finite holonomy group;
 \item[(iv)] there is a closed leaf $L$ with finite holonomy group,
\end{itemize}
then the basic automorphism group ${ A}_B(M,F)$ admits a Lie
group structure of dimen\-si\-on at most $\dim (\mathfrak h) - \dim(\mathfrak k)$,
 and this structure is unique.
\end{cor}

In particular, we have

\begin{cor}\label{c2} If $(M, F)$ is a proper complete Cartan foliation, then
the basic auto\-mor\-phism group ${ A}_B(M,F)$ admits a unique Lie
group structure of dimension at most $\dim (\mathfrak g) - \dim(\mathfrak k)$.
\end{cor}
\begin{rem}\label{rB} I.V. Belko (\cite{Bel}, Theorem 2) stated that
the existence of a closed leaf of a foliation $(M,F)$ with complete
transversally projectable affine connection is sufficient for the fact that
the basic automorphism group $A_B(M,F)$ to admit a Lie group structure. Example~\ref{E4}
shows that this statement is not true in general.
\end{rem}

As it has been indicated above about the existence
of the associated effective Cartan geometry, the investigation of
the basic automorphism  groups  of Cartan foliation is reduced to
foliations which are modelled on effective Cartan geometries.

\begin{defn}\label{opr1} Let $\kappa: \widetilde{M}\to M$ be the universal covering map.
We say that a smooth foliation $(M, F)$
is covered by fibration if the induced foliation $(\widetilde{M}, \widetilde{F})$
is formed by fibres of a locally trivial fibration
$\widetilde{r}: \widetilde{M}\to B.$
\end{defn}

Further we investigate Cartan foliation covered by fibration.

First we describe the global structure entering the holonomy groups
of the Cartan foliations covered by fibrations.

\begin{theorem}\label{Th2}
Let $(M, F)$ be a complete Cartan foliation covered by the fibration
$\widetilde{r}:\widetilde{M}\to B$ where $\widetilde{\kappa}:
\widetilde{M}\to M$  is the universal covering map. Then
\begin{enumerate} \itemsep=0pt
\item[(1)] there exists a regular covering map $\kappa: \widehat{M}\to M$
 such that the induced foliation $\widehat{F}$ is made up of fibres of the locally
trivial bundle $r: \widehat{M}\to B$ over a simply connected Cartan
manifold $(B,\eta)$;
\item[(2)] a group $\Psi$ of automorphisms of the Cartan manifold
$(B,\eta)$ and epimorphism
$\chi: \pi_1(M, x)\to \Psi\,$
of the fundamental group $\pi_1(M, x)$, $x\in M$, onto $\Psi$ are determined;
\item[(3)] for all points $y\in M$ and $z\in\kappa^{-1}(y)$ the restriction
$\kappa|_{\widehat{L}}: \widehat{L}\to L$ to the leaf $\widehat{L} = \widehat{L}(z)$
of the foliation $(\widehat{M},\widehat{F})$ is a regular covering map
onto the  leaf $L = L(y)$, and the group of deck transformations
of $\kappa|_{\widehat{L}}$ is isomorphic to the stationary subgroup
$\Psi_b$ of the group $\Psi$ at the point $b = r(z)\in B$.
Moreover, the subgroup $\Psi_b$ is isomorphic to  the holonomy group
$\Gamma(L, y)$ of the leaf $L$;
\item[(4)] the group of deck transformation of $\kappa: \widehat{M}\to M$ is isomorphic to $\Psi$.
\end{enumerate}
\end{theorem}
\begin{defn} The group $\Psi = \Psi(M, F)$ satisfying Theorem \ref{Th2}
is called the {\it global holonomy group} of the Cartan foliation $(M, F)$
covered by fibration.
\end{defn}

We recall the notion of an Ehresmann connection (subsection \ref{ssEr}).
The following two theorems show that the class of Cartan foliations
covered by fibrations is large.

\begin{theorem}\label{Th3} Let $(B, \eta)$ be any simply connected Cartan manifold
and ${\Psi}$ be any subgroup of the automorphism group $Aut(B,\eta)$
of the Cartan manifold $(B, \eta)$. Then there exists a Cartan foliation
covered by fibration with the global holonomy  group $\Psi$, and
statements of Theorem $\ref{Th2}$ are valid for it.
\end{theorem}

\begin{theorem}\label{ThN} If the transverse Cartan curvature of a complete
Cartan foliation $(M, F)$ is equal to zero, then $(M, F)$ is covered by fibration,
and statements of Theorem $\ref{Th2}$ are valid for it.
\end{theorem}

\begin{rem} The first author proved (\cite{ZhG}, Theorem 5) that any complete
non-Rie\-mannian conformal foliation of codimension $q \geq 3$ is covered by
fibration.
\end{rem}

The application of Theorem 7 proved by the first author in \cite{ZhR}
to Cartan foliations gives us the following interpretation
of the structural Lie algebra of Cartan foliations covered by fibrations.

\begin{theorem}\label{Th4} Let $(M, F)$ be a complete Cartan foliation
covered by the fibration
$\widetilde{r}: \widetilde{M}\to B$
where $\widetilde{\kappa}: \widetilde{M}\to M$  is the universal covering map.
Then the  structural Lie algebra $\mathfrak g_0 = \mathfrak g_0(M, F)$
is isomorphic to the Lie algebra of the Lie group $\overline{\Psi}$,
which is the closure of $\Psi$ in the Lie group $Aut(B, \eta)$,
where $(B, \eta)$ is the induced Cartan  geometry.
\end{theorem}

\begin{cor}\label{c4} Under conditions of Theorem $\ref{Th4}$
the  structural Lie algebra $\mathfrak g_0(M, F)$ is zero if and only if
the global holonomy group $\Psi$ is a  discrete subgroup of the Lie
group $Aut(B, \eta)$ where $\eta$ is the induced Cartan geometry.
\end{cor}

Our next objective is to find a connection between the basic automorphism
group $A_{B}(M, F)$ of Cartan foliation covered by fibration and its global
holonomy group $\Psi$. Application of the foliated bundle over $(M, F)$,
Theorems~$\ref{Th1}, \ref{Th2}$ and $\ref{Th4}$ allow us to accomplish
this task  and to prove the following  statement.

\begin{theorem}\label{Th5}
Let $(M, F)$ be a complete  Cartan foliation  covered by  fibration
$r:\widehat{M}\rightarrow B$ and $(B,\eta)$ is the simply connected
Cartan manifold determined in Theorem~\ref{Th2}. Suppose that the
global holonomy group $\Psi$ is a discrete subgroup in the Lie group
$Aut(B, \eta)$. Let $N(\Psi)$ be the normalizer of $\Psi$ in $Aut(B,
\eta)$. Then the basic automorphism group $A_B(M, F)$ (in the
category of Cartan foliations $\mathfrak{C}\mathfrak{F}$) is a Lie
group which  is isomorphic to an open-closed subgroup of the Lie
quotient group $N(\Psi)/\Psi$, and $\dim(A_{B}(M,
F))=\dim(N(\Psi)/\Psi).$
\end{theorem}

In the following theorem we give sufficient conditions for a Cartan foliation
to satisfy Theorem \ref{Th5} and have the basic automorphisms group
$A_B(M, F)$ isomorphic to the Lie quotient group $N(\Psi)/\Psi$.

\begin{theorem}

\label{Th7}

 Let $(M,F)$ be an $\mathfrak M$-complete Cartan foliation.
If the distribution $\mathfrak M$ is integrable, then

1. The foliation  $(M,F)$ is covered by fibration over the simply
connected Cartan manifold $(B,\eta)$, and  $(M,F)$ is
$(Aut(B,\eta),B)$-foliation.

2. If moreover, the normalizer $N(\Psi)$ of global holonomy group $\Psi$ is
equal to  the centralizer $Z(\Psi)$ of $\Psi$ in the group $Aut(B,\eta)$, then
$$A_B(M, F)\cong N(\Psi)/\Psi.$$
\end{theorem}

\medskip{\noindent\bf Notations\,} We denote by $\mathfrak X(N)$ the Lie algebra of
smooth vector fields on a manifold $N.$ If $\mathfrak M$ is a smooth
distribution on $M$, then $\mathfrak X_{\mathfrak M}(M):=\{X\in\mathfrak X(M)\mid
X_u\in {\mathfrak M}_u\,\,\,\,\,\,\forall u\in M\}$. If in addition $f: K\to M$
is a submersion, then $f^*\mathfrak M$ is the distribution on the  manifold $K$ such that
$(f^*\mathfrak M)_z: = \{X\in T_zK\,|\, f_{*z}(X)\in\mathfrak M_{f(z)}\}$ where $z\in K$.

Let $\mathfrak F\mathfrak o\mathfrak l$ be the category of foliations where morphisms are
smooth maps transforming leaves into leaves.

If $\alpha:G_1\to G_2$ is a group homomorphism, then $Im (\alpha):=\alpha(G_{1}).$
Let $\cong$ be the denotation of a group isomorphism.

Following to \cite{K2} we denote by $P(N, H)$ a principal $H$-bundle over
the mani\-fold~$N$ with the projection $P\to N$.

\section{The category of Cartan foliations}
\label{S2}
\subsection{The category of Cartan geometries}
\label{ss2.2}
We recall here the definition of Cartan geometries (see
\cite{K},\cite{Shar} and \cite{C-S}).

Let $G$ be a Lie group and  $H$ is a closed subgroup of $G$.
Denote  by  $\mathfrak{g}$ and $\mathfrak{h}$
 the Lie algebras of Lie  groups $G$ and $H$ relatively.

\begin{defn}\label{d3}
Let $N$ be a smooth manifold.
{\it A Cartan geometry} on $N$ of type  $(G,H)$ is the principal right
$H$-bundle $P(N,H)$ with the projection $p:P\rightarrow N $ together with
a $\mathfrak{g}$-valued $1$-form $\omega$ on $P$ satisfying the following  conditions:
\begin{enumerate} \itemsep=0pt
\item[({$c_1$})] the map $\omega_{w}:T_{w}P\rightarrow \mathfrak{g}$ is an isomorphism of
vector spaces for every $w\in P$;
\item[($c_{2}$)] $R^{*}_{h}\omega=Ad_{G}(h^{-1})\omega$ for each $h\in H$,
where $Ad_{G}:H\rightarrow GL(\mathfrak{g})$  is the joint representation
of the Lie subgroup $H$ of $G$ in the Lie algebra $\mathfrak{g}$;
\item[($c_{3}$)]  $\omega({A^{*}})=A$ for every $A\in\mathfrak{h}$, where $A^{*}$
is the fundamental vector field determined by $A$.
\end{enumerate}
 The $\mathfrak{g}$-valued form $\omega$ is called a {\it Cartan
 connection form}. This Cartan geometry is denoted  by $\xi=(P(N,H),\omega)$.
 The pair $(N,\xi)$ is called a {\it Cartan manifold}.
\end{defn}
 Let $\xi=(P(N,H),\omega)$ and  $\xi'=(P'(N',H),\omega')$ be two
Cartan geometries with the same structure group $H$.
The smooth map $\Gamma:P\to P'$ is called a morphism from $\xi$ to $\xi'$ if
$\Gamma^{*}\omega'=\omega$ and $R_{a}\circ\Gamma=\Gamma\circ R_{a},\,\,\, a\in H$.
If $\Gamma\in Mor(\xi, \xi')$, then the projection $\gamma:N\to N'$ is defined
such that $p'\circ \Gamma=\gamma\circ p,$  where $p:P\to N$ and $p':P'\to N'$
 are the projections of the corresponding $H$-bundles.
The projection $\gamma$ is called {\it an automorphism of the Cartan
manifold $(N,\xi)$}. Denote by $Aut(N,\xi)$ the full automorphism
group of $(N,\xi)$ and by $Aut(\xi)$ the full automorphism group of
$\xi$. The category of Cartan geometries is denoted  by ${\mathfrak
C }{\mathfrak a}{\mathfrak r}$. Let $A(P,\omega):=\{\Gamma\in Diff
(P)\,|\,{\Gamma^{*}\omega=\omega}\}$ be the automorphism group  of
the parallelizable manifold $(P,\omega)$.

Let $A^{H}(P, \omega):=\{\Gamma\in A(P,\omega)\,|\,
\Gamma\circ R_{a}=R_{a}\circ \Gamma\}$, then $A^{H}(P, \omega)$
is a closed Lie subgroup of the Lie group $A(P,\omega)$ and  $Aut(\xi) = A^{H}(P, \omega)$ is
the automorphism group of Cartan geometry $\xi$.
The Lie group epimorphism $\sigma:A^{H}(P,\omega)\to Aut(N,\xi):\Gamma\mapsto \gamma$
mapping $\Gamma$ to its projection $\gamma$ is defined.

\subsection{Effectiveness of Cartan geometries}\label{sseff}
Remind the notion of {\it effective Cartan geometry}  \cite{Shar}. Consider a
pair Lie groups $(G,H)$, where $H$ is a closed subgroup of $G$. Let
$(\mathfrak{g},\mathfrak{h})$ be the appropriate pair of Lie
algebras. The maximal ideal $\mathfrak{k}$ of the algebra
$\mathfrak{g}$ which is contained in $\mathfrak{h}$ is called {\it
the kernel} of  pair $(\mathfrak{g},\mathfrak{h})$. If
$\mathfrak{k}=0$, then the  pair $(\mathfrak{g},\mathfrak{h})$ is
called { \it effective}.   Maximal normal subgroup $K$ of the group
$G$ belonging to $H$ is called the {\it kernel} of pair $(G,H)$.
As it is known, the Lie algebra of $K$ is equal $\mathfrak{k}$. The
Cartan geometry $\xi=(P(M,H),\omega)$ of the type
$\mathfrak{g}/\mathfrak{h}$
 modelled on pair of the Lie group  $(G,H)$, is called {\it effective}
if the kernel $K$ of the pair $(G,H)$ is trivial. As it was proved in
(\cite{Shar}, Theorem 4.1), the Cartan geometry $\xi=(P(M,H),\omega)$ of type
$\mathfrak{g}/\mathfrak{h}$ is effective if and only if the pair of Lie algebras
$(\mathfrak{g},\mathfrak{h})$ is effective and group
$N:=\{h\in H\,|\, Ad_{G}(h)=id_{\mathfrak{g}}\}$ is trivial.
\begin{rem}\label{r4}
The defined above group epimorphism \-
$\sigma:A^{H}(P, \omega)\to Aut (N,\xi)$ is isomorphism  if and only if the
Cartan geometry $\xi$ is effective.
\end{rem}

\subsection{Determination of foliations by $N$-cocycles}
Let $M$ be a smooth $n$-dimensional manifold.
Let $N$ be a smooth $q$-dimensional manifold the connectivity of which is not assumed.
We will call an $(N, \xi)$-\textit{cocycle} on $M$ a
family $\{U_{i},f_{i},\{\gamma_{ij}\}\}_{ij\in J}$ satisfying the following conditions:
\begin{enumerate} \itemsep=0pt
\item[1)] $\{U_{i}\,|\, i\in J\}$ is a covering of the manifold $M$ by
 open connected  subsets $U_{i}$ of $M$, and $f_{i}: U_{i}\to N$
is a submersion with connected fibres;
\item[2)] if $U_{i}\cap U_{j}\neq\emptyset,\, i,j\in J$, then a isomorphism
$$\gamma_{ij}: {f_{j}(U_{i}\cap U_{j})}\to {f_{i}(U_{i}\cap U_{j})}$$ is defined,
and $\gamma_{ij}$ satisfies the equality $f_{i}=\gamma_{ij} \circ f_{j}$ on $U_{i}\cap U_{j}$;
\item[3)] $\gamma_{ij}\circ \gamma_{jk}=\gamma_{ik}$ if $U_{i}\cap U_{j}\cap U_{k}\neq\emptyset$
for all $x\in U_i\cap U_j\cap U_k$ and $\gamma_{ii}=id_{U_{i}}$, $i, j, k\in J$.
\end{enumerate}

Two $N$-cocycles are called  {\it equivalent} if there exists  an
$N$-cocycle containing both of these cocycles. Let $[\{U_{i},f_{i},\{\gamma_{ij}\}\}_{ij\in J}]$ be the
equivalence class of $N$-cocycles on manifold $M$ containing the
cocycle  $\{U_{i},f_{i},\{\gamma_{ij}\}\}_{ij\in J}$. Denote by $\Sigma$
the set of fibres (or plaques) of all the submersions $f_{i}$ of this equivalence
class. Note, that $\Sigma$ is the base of some new topology $\tau$
in $M$. The linear connected components of the topological space $(M,
\tau)$ form a partition $F:=\{L_{\alpha}\,|\,\,\alpha\in
\mathfrak{J}\}$ of the manifold $M$ which is called { \it the
foliation of the codimension} $q$, $L_{\alpha}$ are called its {\it
leaves} and $M$ is the foliated manifold. It is said that foliation
$(M, F)$ is determined by an $N$-cocycle
$\{U_{i},f_{i},\{\gamma_{ij}\}\}_{ij\in J}$. Further we denote the
foliation by the pair $(M, F)$.

\subsection{Cartan foliations}
Let $N$ be a smooth $q$-dimensional manifold the connectivity of
which is not assumed.  Let $(M, F)$  be a foliation determined by an
$N$-cocycle $\{U_{i},f_{i},\{\gamma_{ij}\}\}_{ij\in J}$.  Let $\xi =
(P(N, H), \omega)$ --- Cartan geometry  of type $\mathfrak
g/\mathfrak h$ with the projection $p:P\to N.$
 For every open subset $V\subset N$  induced  Cartan structure
$\xi_{V}=(P_{V}(V, H), \omega_{V})$ of type $\mathfrak g/\mathfrak h$ such that
$P_{V}:=p^{-1}(V)$ and $\omega_{V}:=\omega|_{P_V}$.

Suppose that for every $\gamma_{ij}:{f_{j}(U_{i}\cap U_{j})}\to {f_{i}(U_{i}\cap U_{j})}$
there exists an isomorphism $\Gamma_{ij}:\xi_{f_{j}(U_{i}\cap U_{j})}\to \xi_{f_{i}(U_{i}\cap U_{j})}$
of the induced Cartan geometries  $\xi_{f_{j}(U_{i}\cap U_{j})}$ and $\xi_{f_{i}(U_{i}\cap U_{j})}$
with the projection $\gamma_{ij}$. Then the foliation $(M, F)$ is referred as
Cartan foliation of type $\mathfrak g/\mathfrak h$ (or type $(G, H)$) in the
since of R.~Blu\-men\-thal~\cite{Bl}. The Cartan geometry $\xi=(P(N,H),\omega)$
is called {\it the transverse Cartan geometry} of $(M, F)$. Also
it is said that the foliation $(M, F)$ \textit{is modelled on the Cartan manifold $(N,\xi)$.}
\begin{rem}
\label{R3} The first author introduced a different notion of Cartan foliation in \cite{Min}
that is equivalent to the notion of Cartan foliation in the sense R.~Blu\-men\-thal
if and only if the transverse Cartan geometry is effective.
\end{rem}

\subsection{Morphisms in the category of Cartan foliations}
Let $(M,F)$ and $(M',F')$ are Cartan foliations defined by an
$(N,\xi)$-cocycle $\eta=\{U_i,f_i, \{\gamma _{ij}\}\}$ and an
$(N',\xi')$-cocycle $\eta'=\{U'_r,f'_r, \{\gamma'_{rs}\}\}$
respectively. All objects be\-long\-ing to $\eta'$ are distinguished by prime.
Let $f\colon M\to M'$ be a smooth map which is a local
isomorphism in the foliation category ${\mathfrak F\mathfrak o\mathfrak l}.$
Hence for any $x\in M$ and $y:=f(x)$ there exist neighborhoods
$U_k\ni x$ and $U'_s\ni y$ from $\eta$ and $\eta'$ respectively,
 a diffeomorphism $\varphi\colon V_k\to V'_s,$ where
$V_k:=f_k(U_k)$ and $V'_s:=f'_s(U'_s),$ satisfying the relations
$f(U_k)=U'_s$ and $\varphi\circ f_k=f'_s\circ f|_{U_k}$. Further we shall
use the following notations: $P_{k}:=P|_{V_{k}},\,\,\,P'_{s}:=P'|_{V'_{s}}$
and $p_{k}:=p|_{P_{k}},\,\,$ $p'_{s}:=p|_{P'_{s}}$

We say that $f$ preserves {\it transverse Cartan structure} of $(M, F)$ if every such
dif\-feo\-mor\-phism $\varphi\colon V_k\to V'_s$ is an isomorphism of
the induced Cartan geometries $(V_k,\xi_{V_k})$ and
$(V'_s,\xi'_{V'_s})$. This means the existence of isomorphism
$\Phi: P_{k}\to P'_{s}$ in the category $\mathfrak C\mathfrak a\mathfrak r$
with the projection $\varphi$, such that  the following diagram
\begin{equation}
\xymatrix {&{P_{k}}\ar[d]_{p_{k}}\ar[rd]^{\Phi}&\\
 M\supset{U_k}\ar[r]^{f_{k}}\ar[rd]^{f|_{U_{k}}} & V_{k}\ar[rd]^{\varphi}  &P'_{s}\ar[d]^{p'_{s}}\\
& M'\supset{U'_{s}}\ar[r]^{f'_{s}}& V'_{s}}\nonumber
\end{equation}
 is commutative. We emphasize that the indicated above isomorphism
$\Phi:P_{k}\rightarrow P'_{s}$  is not unique if the transverse Cartan
geometries  are not effective. This notion is well
defined, i.~e., it does not depend of the choice of neighborhoods
$U_k$ and $U'_k$ from the cocycles $\eta$ and $\eta'.$

\begin{defn}
By a {\it morphism of two Cartan foliations $(M, F)$ and $(M', F')$
} we mean a local diffeomorphism $f:M\to M'$  which transforms leaves to leaves and preserves
transverse Cartan structure. The category $\mathfrak C\mathfrak{F}$
objects of which are Cartan foliations, mor\-phisms are their
mor\-phisms, is called {\it the category of Cartan foliations.}
\end{defn}
\section{The foliated bundle associated with\\ a Cartan foliation}
\subsection{Associated foliated bundles}
The following  statement is important for further, and it was proved by
the first author (\cite{Min}, Proposition 1).

\medskip
\begin{prop}
\label{pr1}
 Let $(M, F)$ be a Cartan foliation in the sense of R. Blumenthal with the transverse Cartan
geometry $\tilde\xi=(\tilde P(N,\tilde H),\tilde\omega)$ of type
$\tilde{\mathfrak g}/{\tilde{\mathfrak h}}$ modeled on a
pair of Lie groups $(\tilde G,\tilde H)$ with kernel~$K$. Then:
\begin{enumerate}
\item [(i)] there exists an effective Cartan geometry $\xi=(P(N,H),\omega)$
of type $\mathfrak g/{\mathfrak h}$, modeled on the pair of Lie groups $(G,H),$ where
$G=\tilde G/ K,$ $H=\tilde H/ K,$ ${\mathfrak g}=\tilde{\mathfrak g}/
{{\mathfrak k}}$, $\mathfrak h=\tilde{\mathfrak h}/{\mathfrak k}$, and
${\mathfrak k}$ is the kernel of the pair of Lie algebras $(\tilde{\mathfrak g},\tilde{\mathfrak h});$
\item  [(ii)] the original foliation $(M,F)$ is a Cartan foliation
 with an effective transverse Cartan geometry $\xi=(P(N,H),\omega)$.
\end{enumerate}
\end{prop}

Proposition $\ref{pr1}$ allows us to construct the foliated bundle
for an arbitrary Cartan foliation in the sense of R. Blumenthal with
noneffective, in general, transverse Cartan geometry $\tilde\xi$.
Because for effective transverse Cartan geometries the notions of
Cartan foliations in the sense of R. Blumenthal and in the sense of
\cite{Min} are equivalent, we apply Proposition 2 from \cite{Min})
to the effective associated transverse Cartan geometry $\xi$
and get Proposition~\ref{pr2}.
Remind that a Cartan foliation of type $\mathfrak g/{\mathfrak 0}$ is named
transversally parallelizable or $e$-foliation.

\medskip
\begin{prop}\label{pr2}
 Let $(M, F)$ be a Cartan foliation modelled on Cartan geometry
$\tilde\xi=(\tilde P(N,\tilde H),\tilde\omega)$ of type $\tilde{\mathfrak g}/{\tilde{\mathfrak h}}$
and $\xi=(P({\mathcal {N}},H),\omega)$ be the associated effective transverse Cartan geometry of
type  $(G,H)$, where  $G=\tilde G/ K,$ $H=\tilde H/ K,$ $K$ is
the kernel of the pair $(\tilde G/\tilde H)$.
Then there exists a principal $H$-bundle with a projection
$\pi :\mathcal{R}\to M$, $H$-invariant foliation $(\mathcal{R},\mathcal {F})$ and
$\mathfrak g$-valued $H$-equivariant $1$-form $\beta$ on ${\mathcal{R}}$ which
satisfy the following conditions:
\begin{enumerate}
\item [(i)] $\beta(A^*)=A$ for any $A\in\mathfrak h$;
\item [(ii)] the mapping $\beta_u:T_u {\mathcal{R}}\to \mathfrak g$ $\forall u\in {\mathcal{R}} $
is surjective, and  $ker(\beta_u) = T_u{\mathcal{F} }$;
\item [(iii)] the foliation $({\mathcal{R}},\mathcal{F})$ is transversally parallelizable;
\item [(iv)]  the Lie derivative $L_X{\beta}$ is equal to zero for every vector field
$X$ tangent to the foliation $({\mathcal{R}}, \mathcal{F}).$
\end{enumerate}
\end{prop}

\begin{defn} The principal $H$-bundle ${\mathcal{R}}(M,H)$ satisfying Proposition \ref{pr2}
is said to be {\it the associated foliated bundle}. The foliation
$({\mathcal{R}}, \mathcal{F})$ is called {\it the associated lifted foliation}
with the Cartan foliation $(M, F).$
\end{defn}

We denote by $\Gamma(L, x),\, x\in L$, the germ holonomy group  of a leaf $L$
of the foliation usually used in the foliation theory \cite{Tam}. Next proposition about different
interpretations of the holonomy groups of any complete Cartan
foliation follows from (\cite{ZhR}, Theorem 4).
\begin{prop}
\label{pr3}
Let $(M, F)$ be a complete Cartan foliation, $L=L(x)$ be an arbitrary
leaf of this foliation and ${\mathcal{L}}={\mathcal{L}}(u),\,\,\,
u\in\pi^{-1}(x),$ be the corresponding leaf of  the lifted
foliation. Then the germ holonomy group $\Gamma(L,x)$ of leaf $L$ is
isomorphic to  each of following two groups:
\begin{enumerate}
\item [(i)]  the group of deck transformations of the regular covering map
$\pi|_{\mathcal{L}}:{{\mathcal{L}}\to L }$;
\item [(ii)] the subgroup $H({\mathcal{L}})=\{a\in H\,|\, R_{a}({\mathcal{L}})=
{\mathcal{L}}\}$ of the Lie group $H$.
\end{enumerate}
If we change $u$ by an other point $\widetilde{u}\in\pi^{-1}(x)$,
then $H({\mathcal{L}})$ is changed by the conjugate subgroup $H(\widetilde{\mathcal{L}})$, where
$\widetilde{\mathcal{L}} = \widetilde{\mathcal{L}}(\widetilde{u})$, in the group $H$.
\end{prop}
Due to Proposition $\ref{pr3}$ the following definition is correct.
\begin{defn}
\label{d7}
The holonomy group  of a complete Cartan foliation $(M, F)$ is cal\-led \textit{discrete}
if the corresponding group $H({\mathcal{L}})$ is a discrete subgroup of the Lie group $H$.
\end{defn}
\subsection{Ehresmann connections for foliations}\label{ssEr}
Let $(M,F)$ be a foliation of codimension $q$ and $\mathfrak M$ be a
smooth $q$-dimensional distribution on $M$ that is transverse to
the foliation $F.$ The piecewise smooth integral curves of the
distribution $\mathfrak M$ are said to be {\it horizontal,} and the
piecewise smooth curves in the leaves are said to be {\it vertical}. A
piecewise smooth mapping $H$ of the square $I_1\times I_2$ to $M$
is called a {\it vertical-horizontal homotopy} if the curve
$H|_{\{s\}\times I_2}$ is vertical for any $s\in I_1$ and the curve
$H|_{I_1\times\{t\}}$ is horizontal for any $t\in I_2.$ In this
case, the pair of paths $(H|_{I_1\times \{0\}},H|_{\{0\}\times
I_2})$  is called the {\it base} of $H.$ It is well known that there
exists at most one vertical-horizontal homotopy with a given base.

A distribution $\mathfrak M$ is called an {\it Ehresmann
connection for a foliation} $(M,F)$ (in the sense of R. A. Blumenthal
and J. J. Hebda \cite{BH}) if, for any pair of paths $(\sigma, h)$ in $M$ with
a common initial point $\sigma(0) = h(0),$ where $\sigma$ is a
horizontal curve and $h$ is a vertical curve, there exists a
vertical-horizontal homotopy $H$ with the base $(\sigma, h).$

For a simple foliation $F,$ i.~e.,
such that it is formed by the fibers of a submersion $r\colon M\to
B,$ a distribution $\mathfrak M$ is an Ehresmann connection for $F$ if
and only if $\mathfrak M$ is an Ehresmann connection for the
submersion $r,$ i.~e., if and only if any smooth curve in $B$
possesses horizontal lifts.

\subsection{Completeness of Cartan foliations}\label{sscom}
Let $(M, F)$
be an arbitrary smooth foliation on a manifold $M$ and $TF$ be the
distribution on $M$ formed by the vector spaces tangent to the
leaves of the foliation $F.$ The vector quotient bundle $TM/TF$ is
called the transverse vector bundle of the foliation $(M,F).$ Let
us fix  an arbitrary smooth distribution $\mathfrak M$ on $M$
that is transverse to the foliation $(M, F),$ i.~e., $T_xM
= T_xF\oplus\mathfrak M_x$, $x\in M$, and identify $TM/TF$ with $\mathfrak M$.

Let $(M, F)$ be a Cartan foliation and $({\mathcal {R}},{\mathcal {F}})$ be
the lifted foliation with $\mathfrak g$-valued $1$-form $\beta$ satisfying
Proposition~\ref{pr2}. It is natural to identify the transverse
vector bundle $T{\mathcal {R}}/T{\mathcal {F}}$ with the distribution
$\widetilde{\mathfrak M}:=\pi^*\mathfrak M$ on ${\mathcal {R}}$.

\begin{defn} The Cartan foliation $(M, F)$
is said to be {\it $\mathfrak M$-complete} if any
transverse vector field $X\in \mathfrak X_{\widetilde{\mathfrak M}}({\mathcal {R}},{\mathcal {F}})$
such that $\beta(X) = \mathrm{const}$ is complete. A Cartan foliation
$(M, F)$ of arbitrary codimension $q$ is said to be {\it complete} if there exists
a smooth $q$-dimensional transverse distribution $\mathfrak M$ on $M$
such that $(M, F)$ is $\mathfrak M$-complete \cite{Min}.
\end{defn}

In other words, $(M, F)$ is an $\mathfrak M$-complete foliation if and only
if the lifted $e$-foliation $({\mathcal {R}},{\mathcal {F}})$ is complete
with respect to the distribution
$\widetilde{\mathfrak M} = \pi^*{\mathfrak M}$ in the sense of L. Conlon~\cite{Con}.

The following statement was proved by the first author
(\cite{Min}, Proposition~3).

\begin{prop}
{\it If $(M, F)$ is an $\mathfrak M$-complete Cartan foliation, then
$\mathfrak M$ is an \-Ehres\-mann connection for this foliation. }
\end{prop}

\subsection{Structural algebras Lie of Lie foliations with dense leaves}

Let $(M, F)$ be a Lie foliation  with dense leaves.  It is the
Cartan foliation  of a type ${\mathfrak g}_{0}/{\mathfrak 0}$. J.~Leslie \cite{Les} was
the first  who observed that the Lie  algebra ${\mathfrak g}_{0}$ of that foliation
is invariant in the category of foliations ${\mathfrak F}{\mathfrak o}{\mathfrak l}$.

\begin{defn}
 The Lie algebra ${\mathfrak g}_{0}$ of the Lie foliation  $(M, F)$
 with dense leaves is called the {\it structural Lie algebra} of $(M, F)$.
\end{defn}
\subsection{Structural Lie algebras of Cartan foliations}
Applying of the relevant results of P. Molino~\cite{Mo} and of L. Conlon~\cite{Con}
on complete $e$-foliations we obtain the following theorem.

\begin{theorem} \label{Th6} Let $(M, F)$ be a complete Cartan foliation
and $({\mathcal {R}},{\mathcal {F}})$ be the associated lifted $e$-foliation.
Then:
\begin{itemize} \item[(i)] the closure of the leaves of the foliation
$\mathcal {F}$ are fibers of a certain locally trivial fibration
$\pi_b\colon{\mathcal {R}}\to W;$
\item[(ii)] the foliation $(\overline{\mathcal {L}},{\mathcal {F} }|_{\overline{\mathcal {L}}})$
induced on the closure $\overline{\mathcal {L}}$ is a Lie foliation with dense leaves
with the structural Lie algebra $\mathfrak g_0$, that is the same for
any $\mathcal {L}\in\mathcal {F}.$
\end{itemize}
\end{theorem}

\begin{defn} The structural Lie algebra $\mathfrak g_0$ of the Lie foliation
$(\overline{\mathcal{L}},{\mathcal {F}}|_{\overline{\mathcal {L}}})$ is called
{\it the structural Lie algebra} of the complete
foliation $(M, F)$ and is denoted by $\mathfrak g_0=\mathfrak g_0(M, F).$
\end{defn}

If $(M, F)$ is a Riemannian foliation on a
compact manifold, this notion coincides with the notion of the
structural Lie algebra in the sense of P. Molino~\cite{Mo}.

\begin{defn} The fibration $\pi_b\colon{\mathcal{R}}\to W$ satisfying
Theorem \ref{Th6} is called the {\it basic fibration} for $(M,F).$
\end{defn}

\section{Basic automorphisms of Cartan foliations}
\subsection{Groups of basic automorphisms of Cartan foliations}

\begin{defn} Let ${A}(M,F)$ be the full automorphism group of a Cartan foliation
$(M, F)$ in the category of Cartan foliation $\mathfrak C \mathfrak F$.
 The group
\begin{equation}
{ A}_L(M,F):=\{f\in{A}(M, F)\mid
f(L_\alpha)=L_\alpha\,\,\,\forall L_\alpha\in F\} \nonumber
\end{equation}
is a normal subgroup of ${ A}(M, F)$ which is called the {\it leaf
automorphism group} of $(M, F).$ The quotient group ${ A}(M,F)/{A}_L(M,F)$ is called
 {\it the basic automorphism group} and is denoted by ${A}_B(M,F).$
\end{defn}

Let us emphasize, that the basic automorphism group ${
A}_B(M, F)$ of a Cartan foliation $(M, F)$ is an invariant of this
foliation in the category ${\mathfrak C}{\mathfrak F}$.

\subsection{Properties of the basic automorphism groups\\ of Cartan foliations}
For a Cartan foliation with effective  transverse Cartan geometry
 Proposition \ref{pr5}   follows from (\cite{ZhR}, Proposition 9).
\begin{prop}
\label{pr5} Let $(M, F)$ be a Cartan foliation  modelled on an
effective Cartan geometry. Let $A^{H}({\mathcal {R}},{\mathcal {F}}):=\{h\in
A({\mathcal {R}}, {\mathcal {F}})\,|\,R_{a} \circ h=h \circ  R_{a} \,\,\,\forall
a \in H\}$,  $A_{L}^{H}({\mathcal {R}}, {\mathcal {F}}):=
\{h\in A_{L}({\mathcal {R}},{\mathcal {F}})\,|\,R_{a}\circ h=h\circ R_{a}\,\,\,
\forall a \in H\}$ and $A^{H}_B({\mathcal {R}}, {\mathcal {F}})$ be the quotient
group $A^{H}({\mathcal {R}}, {\mathcal {F}})/A^{H}_{L}({\mathcal {R}}, {\mathcal {F}})$.

Then there exists the group isomorphism $\delta: A^{H}_B({\mathcal
{R}}, {\mathcal {F}})\to A_{B}(M, F) $ satisfying the commutative diagram
$$\begin{CD}
\xymatrix{
 A^{H}({\mathcal {R}}, {\mathcal {F}})  \ar@{->}[d]_{\alpha^{H}} \ar@{->}[r]^{{\mu}}& A(M, F) \ar@{->}[d]^{\alpha} \\
 A^{H}_{B}({\mathcal {R}}, {\mathcal {F}})   \ar@{->}[r]^{\delta}& A_{B}(M, F).}
\end{CD}$$
where $\alpha^{H}$ and $\alpha$ are the group epimorphisms onto
the indicated quotient groups.
\end{prop}

Assume that the structural Lie algebra $\mathfrak g_0 = \mathfrak g_0(M, F)$ is zero
for a complete Cartan foliation $(M, F)$.
Then the lifted foliation $(\mathcal {R}, \mathcal {F})$ is formed by fibres
of the locally trivial fibration $\pi_b: {\mathcal {R}}\to W$  and the
$\mathfrak{g}$-valued $1$-form $\beta$ on ${\mathcal {R}}$ is determined according
to Proposition $\ref{pr2}$.
In compliance with (\cite{Min}, Proposition 4) the map
\begin{gather}\label{d1}
 W\times H\to W: (w,a)\mapsto \pi_b(R_a(u))\,\,\,\, \forall\, (w,a)\in W\times H,
u\in\pi_b^{-1}(w), \nonumber
\end{gather}
defines a locally free action of the Lie group $H$ on the basic
manifold $W$, and the orbits space $W/H$ is homeomorphic to the leaf
space $M/F$. Identify $W/H$ with $M/F$. Connected components of the
orbits of this action form a regular foliation $(W, F^H)$. The
equality $\pi_b^*\widetilde{\beta}: = \beta$ defines an $\mathfrak
g$-valued $1$-form $\widetilde{\beta}$ on $W$ such that
$\widetilde{\beta}(A_W^*) = A$, where $A_W^*$ is the fundamental
vector field on $W$ defined by $A\in\mathfrak h\subset\mathfrak g.$

Denote by ${A}(W,\widetilde{\beta})$ the Lie group of automorphisms
of the parallelizable manifold $(W,\widetilde{\beta})$, i.e.,
${A}(W,\widetilde{\beta}) = \{f\in Diff(W)\,|\, f^{*}\widetilde{\beta} = \widetilde{\beta}\}$.
Let 
$${ A}^{H}( W,\widetilde{\beta}) = \{f\in{ A}(W,\widetilde{\beta})\,|\, f\circ R_a = R_a\circ f, a\in H \}.$$
Then ${ A}^{H}(W,\widetilde{\beta})$ and its  unity component
${ A}^{H}_{e}(W,\widetilde{\beta})$ are Lie groups as closed subgroups of ${A}(W,\widetilde{\beta})$.

\begin{prop}
\label{pr5.2}
Let $(M, F)$ be a complete Cartan foliation  with an effective
trans\-ver\-se geometry and ${\mathfrak g}_{0}={\mathfrak g}_{0}(M, F)=0 $,  $(W,\widetilde{\beta})$
be the corresponding parallelisable basic  manifold for the lifted foliation
$({\mathcal {R}}, {\mathcal {F}})$, where $W={\mathcal {R}}/ {\mathcal {F}}$.
Then there exists a Lie group monomorphism
\begin{equation}
\nu:A^{H}_{B}({\mathcal {R}}, {\mathcal {F}})\to A^{H}(W,\widetilde{\beta}): h\cdot A^{H}_{L}({\mathcal {R}},
{\mathcal {F}})\mapsto \widetilde{h,}  \nonumber
\end{equation}
where $h\in A^{H}({\mathcal {R}}, {\mathcal {F}})$ and $\widetilde{h}$
is the projection of $h$ with respect to the basic fibration
$\pi_{b}: {\mathcal {R}}\to W$, and $Im(\nu)$ is an open-closed Lie subgroup of $A^{H}(W, \widetilde{\beta})$.

Consequently,
  $\varepsilon=\nu\circ \delta^{-1}: A_{B}(M,F)\to A^{H}(W, \widetilde{\beta})$
is a Lie group monomorphism, and $Im(\varepsilon)$ is an open-closed
Lie subgroup of $A^{H}(W, \widetilde{\beta})$.
\end{prop}
\begin{proof}
By condition $\mathfrak{g}_{0}(M,F)=0$ and the lifted foliation
$({\mathcal {R}}, {\mathcal {F}})$ is formed by fibres  of the submersion $\pi_{b}:{\mathcal {R}}\to W.$
Then every $h~\in~A^H~({\mathcal {R}}, {\mathcal {F}})$ induces $\widetilde{h}\in A^{H}(W, \widetilde{\beta}),$
and the map $\rho: A^H({\mathcal {R}}, {\mathcal {F}})\to A^{H}(W, \widetilde{\beta})$ is defined. It is clear that
$\rho$ is a group homomorphism with the kernel
$Ker (\rho)=A^H_{L}({\mathcal {R}}, {\mathcal {F}})$. As $ A^H_{L}({\mathcal {R}}, {\mathcal {F}})$
is the normal subgroup of $A^H({\mathcal {R}}, {\mathcal {F}})$, there exists a group
monomorphism $\nu:A_{B}^{H}({\mathcal {R}}, {\mathcal {F}})\to A^{H}(W, \widetilde{\beta})$ satisfying
the equality \- ${\rho: = \nu\circ\alpha^H,}$ where
$\alpha^H: A^{H}({\mathcal {R}}, {\mathcal {F}})\to A_{B}^{H}({\mathcal {R}}, {\mathcal {F}})$
is the natural projection onto the quotient group
$A_{B}^{H}({\mathcal {R}}, {\mathcal {F}})=A^{H}({\mathcal {R}}, {\mathcal {F}})/A_{L}^{H}({\mathcal {R}}, {\mathcal {F}})$.

Suppose that $A^{H}(W, \widetilde{\beta})$ is a discrete Lie group,
then $A_{B}^{H}({\mathcal {R}}, {\mathcal {F}})$
is also discrete Lie group and the required statement is true.

Further we assume that $\dim (A^{H}(W))\geq 1$.

Let $\mathfrak{a}$ be the Lie algebra of the Lie group $A^{H}(W, \widetilde{\beta})$.
Let $B^{*}$ be the fundamental vector field defined by $B\in \mathfrak a$.
Hence $X: = B^*$ is a complete vector field on $W$, which defines an $1$-parameter group
$\phi_{t}^{X}, \,\,\, t\in(-\infty,+\infty)$,
of transformations from $A^{H}(W, \widetilde{\beta})$.

Let $f$ be any element from the identity component
$A^H_e(W, \widetilde{\beta})$ of the Lie group $A^{H}(W, \widetilde{\beta})$.
Then there exists $B\in\mathfrak a$ and $t_0\in
(-\infty,+\infty)$ such that $f = \phi_{t_0}^{X}$ where $X = B^*$.
Since $\pi_b: {{\mathcal {R}}\to W}$ is the submersion with the Ehresmann
connection $\widetilde{\mathfrak M}$, where $\widetilde{\mathfrak M}
=\pi^{*}{\mathfrak M}$, there exists the unique vector field $Y\in\mathfrak
X_{\widetilde{\mathfrak M}}(\mathcal {R})$ such that $\pi_{b^*} Y = X$. The
completeness of the vector field $X$ implies the
completeness of the vector field $Y$.  Hence
$Y$ defines a $1$-parameter group $\psi_{t}^{Y}, \,\,\,
t\in(-\infty,+\infty)$,
 of diffeomorphisms of the manifold $\mathcal {R}$.

 Let us shows that
$\psi_{t}^{Y} \in A^H_{e}({\mathcal {R}},{\mathcal {F}})$ for all $t\in(-\infty,+\infty)$,
i.e., we have to check the validity  of the following facts:
1) the map $\psi_{t}^{Y},\,t\in(-\infty,+\infty)$, is an  isomorphism of
$({\mathcal {R}},{\mathcal {F}})$ in the category $\mathfrak{F}\mathfrak{o}\mathfrak{l}$;
2) $L_{Y}\beta=0$; 3) $L_{Y}A^{*}=0$ for all $A\in \mathfrak h$.

1) The equality $\pi_{b^{*}}Y = X$ implies the relation
$\pi_{b}\circ \psi_{t}^{Y}=\psi_{t}^{X}\circ \pi_{b}$ for any fixed  $t~\in~(-\infty,~\infty)$,
hence $\psi_{t}^{Y}(\pi_{b}^{-1}(v))=\pi_{b}^{-1}(\psi_{t}^{X}(v))$ for all $v\in W,$
and $\psi_{t}^{Y}$ is the isomorphism the lifted foliation $({\mathcal {R}},{\mathcal {F}})$
in the category $\mathfrak F \mathfrak o \mathfrak l$.

2) Take arbitrary $u\in {\mathcal {R}}$ and $Z_{0}\in \widetilde{\mathfrak M}_{u}$.
There is the unique vector field $Z\in \mathfrak X_{\widetilde{\mathfrak M}}({\mathcal {R}})$
such that $Z|_{u}=Z_{0}$ and $\beta(Z)=\beta(Z_{0})=const.$ Put $Z_{W}:=\pi_{b^{*}}Z$
and apply the following formula \cite{K}
\begin{equation}
(L_{X}{\widetilde{\beta}})(Z_{W})=X(\widetilde{\beta}(Z_{W}))-\widetilde{\beta}([X, Z_{W}]).\label{F1}
\end{equation}

The relation $\beta=\widetilde{\beta}\circ\pi_{b^{*}}$ implies that
$\widetilde{\beta}(Z_{W})=\beta(Z_{0})=const,$ so $X(\widetilde{\beta}(Z_{W}))=0$.
By the choice of $X$,  $\phi_{t}^{X}\in A^{H}(W,\widetilde{\beta})$,
therefore  we have $L_{X}\widetilde{\beta}=0.$ Hence the equality \eqref{F1} gives
\begin{equation}
\label{F2}
\widetilde{\beta}([X, Z_{W}])=0.
\end{equation}
In the formula
\begin{equation}
(L_{Y}{\beta})(Z)=Y(\beta(Z))-\beta([Y, Z]).\label{F3}
\end{equation}
the  first term $Y(\beta(Z))=0$, because $\beta(Z)=const.$ The relations
$\beta=\widetilde{\beta}\circ \pi_{b^{*}}$ and \eqref{F2} imply the following of equalities:
$$\beta([Y,Z])=\widetilde{\beta}(\pi_{b^{*}}[Y,Z])=\widetilde{\beta}([\pi_{b^{*}}Y,\,\pi_{b^{*}}Z])
=\widetilde{\beta}([X,Z_{W}])=0.$$
Therefore \eqref{F3} implies that $(L_{Y}\beta)(Z)=0$. Thus, $L_{Y}\beta=0$.

3) Denote by $(W, F^{H})$ the foliation formed by  the connected  components
of orbits of the action $\Phi^{W}$ the defined above of $H$ on $W$.
Let $({\mathcal {R}}, {\mathcal {F}}^{H})$ be the foliation formed
by the connected components  of orbits of the Lie group $H$ on ${\mathcal {R}}$.

At any point $u\in {\mathcal {R}}$ there is an neighbourhood ${\mathcal {W}}$
foliated with respect to both foliations $({\mathcal {R}}, {\mathcal {F}})$ and
$({\mathcal {R}}, {\mathcal {F}}^{H})$ which meets each leaf of these foliation
in at most one  connected subset.  We can suppose  that the basic fibration
$\pi_{b}:{\mathcal {R}}\to{{W}}$ is the trivial in the neighbourhood
$\pi_{b}^{-1}({\mathcal {V}}),$ where ${\mathcal {V}}:=\pi_{b}({\mathcal {W}})$.
Put $U=\pi({\mathcal {W}})$. Let $r:U\to U/(\mathcal{F}|_{U})$ and
$s:{\mathcal {V}}\to{\mathcal {V}}/({\mathcal {F}}^{H}|_{\mathcal {V}})$ be the
quotient maps. We can identify $ U/(\mathcal{F}|_{U})$ and ${\mathcal {V}}/({\mathcal {F}}^{H}|_{\mathcal {V}})$
with the manifold $V$ such that the diagram
\begin{equation}
\begin{CD}
\xymatrix{
 {\mathcal {W}} \ar@{->}[d]_{\pi} \ar@{->}[r]^{\pi_{b}}&{\mathcal {\mathcal {V}}} \ar@{->}[d]^{s} \\
U  \ar@{->}[r]^{r}& V,}
\label{D1}
\end{CD}
\end{equation}
where the restriction  of $\pi$ and $\pi_{b}$ onto ${\mathcal {W}}$
are denote by the same letters,  is commutative.  Without lost the generality,
we can assume  that ${\mathfrak M}|_{U}$ is an Ehresmann connection for
the submersion $r$ and $\widetilde{\mathfrak M}|_{\mathcal {W}}$
is an Ehresmann connection  for the submersion $\pi_{b}$.

 Let $A^{*}_{W}$ be the fundamental vector field on $W$ defined by $A\in \mathfrak{h}$.
 Since the action $\Phi^{H}$ of the Lie group $H$ on $W$ is defined  any element
$A\in\mathfrak {h}$ defines $1$-parametric group  tangent vector field to which
is called  the fundamental vector field and  is denoted by $A^{*}_{W}$. By the
choice of $X:=B^{*}$ on $W$ for any $A\in\mathfrak h$ we have the equality
$L_{X}A^{*}_{W}=0, $ i.e., $[A^{*}_{W},X]=0$. Since the fundamental vector fields
$A^{*}_{W}$ span the tangent spaces to the leaves of the foliation $(W, F^{H})$,
it is not difficult  to cheek that $X$ is the foliated vector field for this foliation.
Hence the vector field $X_{V}:=s_{*}X|_{\mathcal {V} }$ is well defined.
There is the unique vector field $Y_{U}\in\mathfrak{X}_{\mathfrak M}(U)$ such that
$r_{*}Y_{U}=X_{V}$. In other words, $Y_{U}$ is the $\mathfrak{M}$-horizontal lift of $X_{V}$.
 The commutative diagram \eqref{D1} implies the relation $\pi_{*}Y_{\mathcal {W}}=Y_{U}$,
hence $Y$ is a foliated vector field with respect to the foliation
$({\mathcal {R}}, {\mathcal {F}}^{H})$. Therefore,
\begin{equation}
\label{F4}
[A^{*}, Y]\in {\mathfrak X}_{{\mathcal {F}}^{H}}({\mathcal {R}}).
 \end{equation}
Due to the equalities $\beta(A^{*})=\widetilde{\beta}(A^{*}_{W})=A,$ the
vector field $A^{*}$ is foliated with respect to $({\mathcal {R}}, {\mathcal {F}})$.
So we have the following chain of equalities
$$\pi_{b^{*}}[A^{*},Y]=[\pi_{b^{*}}A^{*}, \pi_{b^{*}}Y]=[A^{*}_{W}, X]=0,$$
hence,
 \begin{equation}\label{F5}
 [A^{*}, Y]\in {\mathfrak X}_{{\mathcal {F}}}({\mathcal {R}}).
 \end{equation}
 The relations \eqref{F4} and \eqref{F5} imply the equality $ [A^{*}, Y]=0$ for all
$A\in\mathfrak h.$ This ends the check that $\phi_{t}^{Y}\in {{A}}_{e}^{H}({\mathcal {R}},{\mathcal {F}})$,
and consequently $\phi_{t}^{X}\in {{A}}_{e}^{H}(W, \widetilde{\beta})$.
Thus, we proved the inclusion ${A}_{e}^{H}(W, \widetilde{\beta})\subset Im(\rho) = Im(\nu)$.
Therefore $Im(\nu)$ is an open-closed Lie subgroup of ${A}^{H}(W, \widetilde{\beta})$

Therefore, $\varepsilon=\nu\circ \delta^{-1}: A_{B}^{H}({\mathcal {R}},
{\mathcal {F}})\to {A}^{H}(W, \widetilde{\beta}):
\widehat{f}\cdot A_{L}^{H}({\mathcal {R}}, {\mathcal {F}})\mapsto f$
is the monomorphism of Lie groups, and
$Im(\varepsilon)=\varepsilon(A_{B}^{H}({\mathcal {R}},{\mathcal {F}}))$ is
an open-closed Lie subgroup of the Lie group of ${A}^{H}(W, \widetilde{\beta})$.
\end{proof}

\subsection{Proof Theorem \ref{Th1}}
\label{Prove1}
Let $(M, F)$ be a Cartan foliation modelled on a Cartan geometry
$\widetilde{\xi}=(\widetilde{P}(N,H),\widetilde{\omega})$
 of type $(\widetilde{G},\widetilde{H})$ and the Lie group $K$ be the kernel of
the pair Lie groups  $(\widetilde{G},\widetilde{H})$, $\mathfrak k$ be
the Lie algebra of $K$. Then the associated effective
 Cartan geometry $\xi = (P(M,H),\omega)$ of type $(G,H)$,
where $G=\widetilde{G}/{K}$, $H = \widetilde{H}/{K}$, is defined. Here
$\omega$ is the ${\mathfrak g}$-valued $1$-form on $P$, where
$\mathfrak g = \widetilde{\mathfrak g}/{\mathfrak k}$.
According to Proposition \ref{pr2} the associated foliated bundle
${\mathcal {R}}(M,H)$ with the lifted foliation $({\mathcal {R}}, {\mathcal {F}})$
and the projection $\pi: {\mathcal{R}}\to M$ are defined.

If $\dim (A^{H}(W, \widetilde{\beta}))=0$, then according to Proposition
$\ref{pr5.2}$, $A_{B}(M, F)$ is a discrete Lie group, hence the estimates
\eqref{oz1} and \eqref{oz2} are valid.

Suppose now that $\dim (A_{B}(M, F))\geq 1$.

Denote by $A_{B}(M, F)_{e}$ the unite component of the Lie group $A_{B}(M, F)$.
According to Proposition~\ref{pr5.2}, $\varepsilon|_{ A_{B}(M, F)_{e}}: A_{B}(M, F)_{e}
\to{A}^H_{e}(W, \widetilde{\beta})$ is the group isomorphism.
Therefore the basic automorphism group $A_B(M, F)$ admits a Lie group structure
of the dimension not more then $\dim(W) = \dim(\mathfrak g)$.
Because $\frak g = \widetilde {\frak g} /
{\frak k}$ we have $\dim(\frak g) = \dim(\widetilde {\frak g} ) -
\dim(\frak k)$. Hence the dimension of the Lie group $A_B(M, F)$
satisfies the following inequality $\dim(A_{B}(M, F)) \leq
\dim(\widetilde {\frak g} ) - \dim(\frak k).$ As the Lie group
$A_B(M, F)_e$ is realized as a closed subgroup of the automorphism
Lie group ${A}(W,\widetilde{\beta})$ of a parallelizable manifold, then it
admits a unique topology and a unique smooth structure that make it into a Lie group
(\cite{BZ}, Proposition 1). The same is valid for the group $A_B(M, F)$.

$(a)$ Let $s:W\rightarrow W/H$ be the projection onto the orbit space.
Assume now that there exists an isolated proper leaf $L$ of the
foliation $(M,F).$  Let $x\in L$, $u=\pi^{-1}(x)$ and $w =\pi_{b}(u)\in W$.
 Observe, that any automorphism  of a foliation transforms a proper leaf
to the corresponding proper leaf.   Then $q(L) = s(w)$ and
the orbit $A^{H}_e(W, \widetilde{\beta})\cdot w$ belongs to $s^{-1}(s(w))$.
Consequently we have $\dim(A_{B}(M, F)) = 
\dim(A^H_{e}(W,\widetilde{\beta})\cdot w) \leq \dim(H)= \dim(\mathfrak{h})
= \dim(\widetilde{\frak  h}) - \dim(\frak k)$.

 Thus $\dim({A_B}(M, F))\leq \dim(\widetilde{\frak h}) - \dim({\frak k})$.

Suppose now that the  set of proper leaves of $(M, F)$ is countable (non\-empty).
Consider any $1$-parametric group $\varphi_{t}$, $t\in(-\infty,\,+\infty)$
from the Lie group $A^H(W, \widetilde{\beta}) \cong A_B(M,F)_{e}$.
Let $L=L(x)$ be any leaf, $u=\pi^{-1}(x)$ and $w =\pi_{b}(u)\in W$. Let
$w\cdot H$ be the orbit of $w$ relatively $H$. Since for any fixed $t$ the
automorphism $\varphi_{t}$ transforms  a proper leaf $L$ to the proper leaf
$\varphi_{t}(L)$ the countability of  the set proper leaves implies that
$\varphi_{t}(w\cdot H)= w\cdot H$. Hence, by analogy with the previous
case we have the estimate \eqref{oz2}.

$(b)$ Now we suppose that the set of proper leaves $\{L_{n}\,|\,n\in \mathbb{N}\}$ be
countable and dense. Let $x_n\in L_n$, $u_n=\pi^{-1}(x_n)$ and
$w_n =\pi_{b}(u_n)\in W$. Let us assume that $\dim(A_{B}(M, F))~\geq 1$.
Let $\varphi_{t}$, $t\in(-\infty,\,+\infty)$, be any $1$-parametric subgroup
of the automorphism group $A^H(W, \widetilde{\beta}) \cong A_B(M,F)$.
As it was proved above, it is  necessary $\varphi_{t}(w_n\cdot H) = w_n\cdot H$
for all $t\in(-\infty,\,+\infty)$ and $n\in\mathbb N$.
Remark that the leaf space $M/F$ homeomorphic the orbit space $W/H$.
Denote by $\widetilde{\varphi}_{t}$ the induced $1$-parametric group of
homeomorphisms of the leaf space $M/F$. Therefore, for each
$t\in(-\infty,\,+\infty)$, $t\neq 0$, the homeomorphism $\widetilde{\varphi}_{t}$
has dense subset $\{[L_n]\,|\, [L_n] = s( w_n\cdot H), \, n\in \mathbb N\}$
of fixed points in the leaf space $M/F = W/H$.

Due to continuity and openness of the projection $q: M\to M/F$,
the leaf space $M/F$ is a first-countable space, that is every its point has
a countable neighbourhood basis. Then $\widetilde{\varphi}_{t}$ is sequentially continuous.
Therefore the existence of dense subset of fixed points of
the homeomorphism $\widetilde{\varphi}_{t}$ implies $\widetilde{\varphi}_{t} = id_{M/F}$.
Hence $\varphi_{t} = id_{W}$ that contradicts to the assumption.

Thus, $\dim (A_{B}(M, F))=0$ and $(\ref{oz3})$ is proved.

\subsection{Proof of Corollary \ref{c1}} Observe that  the existence
of a proper leaf $L$ with discrete holonomy group guarantees the
equality $\frak g_0(M, F)=0.$

Remark that any closed leaf of a foliation is proper and each finite
holonomy group is a discrete one. Hence we have implications
 $(iv)\Rightarrow (iii) \Rightarrow (ii) \Rightarrow (i).$
Thus, applying Theorem~\ref{Th1} we get the required assertion.

\subsection{Proof of Corollary \ref{c2}}
It is well known that any foliation has leaves without holonomy.
Therefore, the Corollary \ref{c2} follows from the item $(iii)$ of
Corollary \ref{c1}.

\section{The structure of Cartan foliations covered by\\ fibrations}

\subsection{$(G, B)$-foliations} \label{ss7.1}
Let $B$ be a connected smooth manifold and $G$ be a Lie group of
diffeomorphisms of $B$. The group $G$ is said\textit{ to act
quasi-analytically on $B$} if, for any open subset $V$ in $B$ and an
element $g\in G$ the equality $g|_V = id_V$ implies
$g = e$, where $e$ is the identity transformation of $B$.

\begin{defn} Assume that the Lie group  $G$ of diffeomorphisms of
a manifold $B$ acts on $N$ quasi-analytically. A foliation $(M, F)$
defined by an $B$-cocycle $\{U_i,f_i,\{\gamma_{ij}\}\}_{i,j\in J}$
is called a $(G, B)$-foliation if for any $U_i\cap U_j
\neq\emptyset$, $i,j\in J$, there is an element $g\in G$
such that $\gamma_{ij} = g|_{f_j(U_i\cap U_j)}$.
\end{defn}

\subsection{Proof of Theorem \ref{Th2}}
The following lemma will be useful for us.

\begin{lem}\label{L1} Let $\eta = (P(B,H),\beta)$ be an effective Cartan geometry
on a connected manifold $B$ and $\Phi$ be a group of automorphisms
of $(B,\eta)$. Then the group $\Phi$ acts quasi-analytically on $B$.
\end{lem}
\begin{proof} Suppose that there are $\gamma\in\Phi$ and an open subset
$U\subset B$ such that $\gamma|_U = id_U$. Then there exists a
unique $\Gamma\in Aut(\xi)$  lying over $\gamma$. Let $p:P\to B$ be
the projection of the $H$-bundle $P(B,H)$. Observe that any
connected component of $P$ is a connected component of some point
$v\in p^{-1}(U)$, i.~e. it may be represented in  the form $P_v$. The
effectiveness of the Cartan geometry $\eta$ implies $\Gamma|_{p^{-1}(U)} =
id_{p^{-1}(U)}$. Hence $\Gamma$ preserves each connected component
$P_v$ of $P$. Because $\Gamma$ is an isomorphism of the connected
parallelizable manifold $(P_v,\beta|_{P_v})$ and $\Gamma(v) = v$,
then it is necessary $\Gamma|_{P_v} = id_{P_v}$. Therefore $\Gamma =
id_P$ and $\gamma = id_B.$

Thus, the group $\Phi$ acts quasi-analytically on $B$.
\end{proof}

Suppose that a Cartan foliation $(M, F)$ modelled on the effective
Cartan geometry $\xi = (P(N,H),\omega)$ is covered by a fibration
$\widetilde{r}: \widetilde{M}\to B$, where $\widetilde{\kappa}:
\widetilde{M}\to M$ is the universal covering map. The fibration
$\widetilde{r}: \widetilde{M}\to B$ has connected fibres and simply
connected space $\widetilde{M}$. Therefore, due to the application of
the exact homotopic sequence for this fibration we obtain that the base
manifold $B$ is also simply connected.

For an arbitrary point $b\in B$ take $y\in\widetilde{r}^{-1}(b)$
and $x=\widetilde{\kappa}(y)$. Without loss generality we assume
that there is a neighbourhood $U_i$, $x\in U_i$, from the $N$-cocycle
$\{U_i,f_i,\{\gamma_{ij}\}\}_{i,j\in J}$ which defines $(M, F)$ and
a neighbourhood $\widetilde{U}_i$, $y\in\widetilde{U}_i$, such that
$\widetilde{\kappa}|_{\widetilde{U}_i}: \widetilde{U}_i\to U_i$ is a
diffeomorphism.

Let $\widetilde{V}_i: = \widetilde{r}(\widetilde{U}_i)$. Then there
exists a diffeomorphism $\phi: \widetilde{V}_i\to V_i$  satisfying
the equality $\phi\circ\widetilde{r}|_{\widetilde{U}_i} =
f_i\circ\widetilde{\kappa}|_{\widetilde{U}_i}$. The diffeomorphism
$\phi$ induces the Cartan geometry $\eta_{\widetilde{V}_i}$ on
$\widetilde{V}_i$ such that $\phi$ becomes the isomorphism
$(\widetilde{V}_i,\eta_{\widetilde{V}_i})$ and $(V_i,\xi_{V_i})$ in
the category $\frak C\frak a \frak r$ of Cartan geometries. The
direct check shows that by this way we define the Cartan geometry
$\eta$ on $B$, and $\eta|_{\widetilde{V}_i} =
\eta_{\widetilde{V}_i}$, $i\in J$.

Let us fix points $x_0\in M$ and
$y_0\in\widetilde{\kappa}^{-1}(x_0)\in\widetilde{M}$. Then the fundamental group
$\pi_1(M,x_0)$ acts on the universal covering space $\widetilde{M}$
as a deck transformation group $\widetilde{G}\cong\pi_1(M,x_0)$ of
$\widetilde{\kappa}$. Since $\widetilde{G}$ preserves the inducted
foliation $(\widetilde{M},\widetilde{F})$ formed by fibres of the
fibration $\widetilde{r}: \widetilde{M}\to B$, then every
$\widetilde{\psi}\in\widetilde{G}$ defines $\psi\in Diff(B)$
satisfying the relation $\widetilde{r}\circ\widetilde{\psi} =
\psi\circ\widetilde{r}$. The map $\chi: \widetilde{G}\to\Psi:
\widetilde{\psi}\to\psi$ is the group epimorphism. Observe that $\widetilde{G}$
is a subgroup of the automorphism group $A(\widetilde{M}, \widetilde{F})$ of
$(\widetilde{M}, \widetilde{F})$ in the category $\frak C\frak F$.
Therefore $\Psi$ is a subgroup of the automorphism group
$Aut(B,\eta)$ in the category  of Cartan
geometries $\frak C\frak a\frak r$. The kernel $ker(\chi)$ of $\chi$
determines the quotient manifold $\widehat{M}: =
\widetilde{M}/ker(\chi)$ with the quotient map
$\widehat{\kappa}: \widetilde{M}\to\widehat{M}$ and the quotient
group $\widehat{G}: = \widetilde{G}/ker(\chi)$ such that $M
\cong\widehat{M}/\widehat{G}$. The quotient map $\kappa: \widehat{M}\to M$ is
the required regular covering map, with $\widehat{G}$ acts on $\widehat{M}$ as
a deck transformation group. The map $\widehat{G}\to\Psi:
\widetilde{\psi}ker(\chi)\mapsto\chi(\widetilde{\psi})$, $\widetilde{\psi}\in\widetilde{G}$,
is a group isomorphism.

Remark that the induced foliation $(\widehat{M},\widehat{F})$, $\widehat{F} = \kappa^{*}F$,
is covered by a foliation $r: \widehat{M}\to B$ such that $\widetilde{r} = r\circ \widehat{\kappa}$.

Now the assertion $(3)$ of Theorem \ref{Th2} is easy proved with the application of Lemma~\ref{L1}.

\subsection{Suspended foliations}

Suspension foliation was introduced by A. Haefliger. Let $Q$ and $T$
 be smooth connected manifolds. Denote by $\rho : \pi_1 (Q, x)\to Diff(T)$ a group homomorphism.
Let $G:=\pi_1 (Q, x)$ and $\Phi:=\rho (G)$. Consider a universal covering map
$\widetilde{p} : \widetilde{Q} \to Q$. A right action of the group $G$ on product
of manifolds  $\widetilde{Q}\times T$ is defined as follows:
\begin{equation}
\Theta \, : \,\widetilde{Q}  \times T \times G\to \widetilde{Q}\times T
 : \, (x,t, g)\to (x\cdot g ,\, \rho (g^{-1})(t)),\nonumber
\end{equation}
where the covering transformation
$\widetilde{Q}\to \widetilde{Q}: x\to x\cdot g$ is induced by an element $g\in G$.

The quotient manifold $M: = (\widetilde{Q}  \times T)/G$ with the canonical projection
$$f_0: \widetilde{Q}\times T\to M ={(\widetilde{Q}\times T) \,/ G}$$ are determined.

 Let $\Theta_g:=\Theta |_{\widetilde{Q}\times\{t\}\times\{g\}}$. Since
$\Theta_g(\widetilde{Q}\times\{t\})=\widetilde{Q}\times\rho(g^{-1})(t)\,$ $\forall t\in T$,
then the action of the discrete group $G$ on $(\widetilde{Q}\times T)$
preserves the trivial foliation
$F:=\{\widetilde{Q}\times \{t\} \,|\, t\in T \}$ of the product
$\widetilde{Q}\times T$. Thus the projection $f_0: \widetilde{Q}\times
T\to M$ induced on the  $M$ of the
smooth foliation $F$. The pair $(M, F)$ is called {\it a suspended
foliation}  and is denoted by $Sus (T,Q,\rho)$. We accentuate that $(M,F)$
is covered  by the trivial fibration  $\widetilde{Q}\times T\to T.$

\subsection{Proof of Theorem \ref{Th3}}

Let $\eta$ be an effective Cartan geometry on a simply connected
manifold $B$ and $\Psi$ be any countable subgroup of the automorphism Lie group
$Aut(B,\eta)$. Therefore $\Psi$ has not more then countable generations
$\psi_{i}, \,\,i\in \mathbb{N}$. Denote by $\mathbb{R}^{2}$ the usual plane
and $C:=\{(n,0)\in\mathbb{R}^{2}\,|\,n\in\mathbb{N}\}$. Consider the set
$Q:=\mathbb{R}^{2}\setminus C$. Then the fundamental group $\pi_{1}(Q)$ of
$Q$ is the free group  with countable set of generators $\{\alpha_{i}\,|\,
i\in \mathbb{N}\}$. Define a group homo\-morphism $\rho:\pi_{1}(Q)\to Aut(B, \eta)$
by the following equalities on generations $\rho(\alpha_{i})=\psi_{i}$ for every $i\in \mathbb{N}$.

Then we construct the suspended foliation $(M, F):=Sus(B, Q, \rho)$ which is the
$2$-dimensional $(Aut(B,\eta),B)$-foliation covered by the trivial
fibration $\mathbb{R}^{2}\times B\to B.$ Because $Im(\rho)=\Psi$ the group $\Psi$
is the global holonomy group of the Cartan foliation $(M, F)$.

\subsection{Proof of Theorem \ref{ThN}} Let $(M, F)$ be a complete Cartan
foliation of type $(G,H)$ whose transverse curvature is equal to zero.
Then $(M, F)$  is modelled on the Cartan geometry $\xi_0 = (G(G/H,H),\beta_0)$
where $\beta_0$ is the Maurer--Cartan $\frak g$-valued $1$-form on the Lie group $G$.
Hence $(M, F)$ is $(Aut(\xi_0), G/H)$-foliation. According to (\cite{Min}, Proposition 3)
the completeness of $(M, F)$ implies the existence of an Ehresmann connection
for this foliation. As it is well known (\cite{ZhG}, Theorem 2), any
$(G,N)$-foliation with an Ehresmann connection is covered by a fibration.
Therefore, $(M, F)$ is the Cartan foliation covered by a fibration and all
statements of Theorem~\ref{Th2} are valid for it.

\section{The structure of basic automorphism groups of\\ foliations covered by fibrations}
\subsection{Properties of regular covering maps}
\begin{defn}
Let $f:M\rightarrow B$ be a submersion. It is said that $\widehat{h}\in Diff (M)$
lying over $h\in Diff (B)$ relatively $f$ if $h\circ f=f\circ \widehat{h}$.
\end{defn}

Let $\widetilde{\kappa}:\widetilde{{K}}\to K$  be the universal covering map, where $K$ and $\widetilde{{K}}$
are smooth manifolds. By analogy with Theorem 28.7 in \cite{Bus}, it is easy to show
that for any $h\in Diff(K)$ there exists $\widetilde{h}\in Diff(\widetilde{K})$
lying over $h$. It is well known that this fact is not true for
regular covering maps in general.
Proposition \ref{pr8.2} solves the problem of lifting of transformations relatively
arbitrary regular covering maps.

\begin{prop}\label{pr8.2}
Let ${\kappa}:\widehat{K}\to K$ be a smooth regular  covering map with
the deck transformation group $\Gamma$ and $\widetilde{\kappa}:\widetilde{K}\to K$
be the universal covering map with
the deck transformation group $\widetilde\Gamma$. Then
\begin{itemize}
\item[(1)] A diffeomorphism $\widehat{h}\in Diff(\widehat{K})$
lies over some diffeomorphism  $h\in Diff(K)$ if and only if it satisfies the equality
$\widehat{h}\circ\Gamma=\Gamma\circ \widehat{h}$.
\item[(2)] For  $h\in Diff(K)$ there exists $\widehat{h}\in Diff(\widehat{K})$
lying over $h$ if and only if there is $\widetilde h$ lying over $h$ relatively
$\widetilde{\kappa}:\widetilde{K}\to K$ such that
$\widetilde{h}\circ\widehat{\Gamma} = \widehat{\Gamma}\circ\widetilde{h}$ and
$\widetilde{h}\circ\widetilde {\Gamma}=\widetilde {\Gamma}\circ \widetilde{h}$,
where $\widehat{\Gamma}$ is the deck transformation group of the universal
covering map $\widehat{\kappa}:\widetilde{{K}}\to \widehat{K},$ with
$\widehat{\Gamma}$ is a normal subgroup of $\widetilde{\Gamma}$
\item[(3)] The set of all diffeomorphisms lying over $id_{K}$ relatively
${\kappa}:\widehat{K}\to K$ is coincided with the deck transformation group of
${\Gamma}\cong \widetilde{\Gamma}/\widehat\Gamma$.
\item[(4)] The subset of $h\in Diff(K)$ for which there exists a diffeomorphism
$\widehat{h}$ of $\widehat{K}$ lies over $h$ forms a group which is isomorphic
to the quotient group $N(\Gamma)/\Gamma$.
\item[(5)]  Let $G$ be a group of diffeomorphisms of the manifold $K$
such that for every $g\in G$ there exists $\widehat{g}\in
Diff(\widehat{K})$ lying over $g$ relatively $k:\widehat{K}\to K$.
Then the full group $\widehat{G}$ of $\widehat{g}\in Diff(\widehat{K})$ lying over
transformations from $G$ is isomorphic to the  quotient
group $\widehat{G}/\Gamma.$

\end{itemize}
\end{prop}
\begin{proof} $(1)$ Let $\widehat{h}\in Diff(\widehat{K})$ lies over $h$ relatively
$\kappa:\widehat{K}\to K,$ i.e. $h\circ\kappa = \kappa\circ\widehat{h}.$
Then there exists $\widehat{h}^{-1}\in Diff(\widehat{K})$ and
$h^{-1}\circ\kappa = \kappa\circ\widehat{h}^{-1},$ i.e., $\widehat{h}^{-1}$
lies over $h^{-1}$ relatively $\kappa$.

Let $\gamma$ be any element from $\Gamma$, then, according to the assumption $(2)$,
$\kappa\circ\gamma = \kappa$. Using this we get the following chain of equalities
$\kappa\circ(\widehat{h}\circ\gamma\circ\widehat{h}^{-1}) =
(\kappa\circ\widehat{h})\circ\gamma\circ\widehat{h}^{-1} =
 (h\circ\kappa)\circ\gamma\circ\widehat{h}^{-1} = h\circ(\kappa\circ\gamma)\circ\widehat{h}^{-1} =
h\circ\kappa\circ\widehat{h}^{-1} = (h\circ\kappa)\circ\widehat{h}^{-1} =
(\kappa\circ\widehat{h})\circ\widehat{h}^{-1} = \kappa.$ Due to $(2)$
this implies $\widehat{h}\circ\gamma\circ\widehat{h}^{-1}=\gamma'\in\Gamma$
and $\widehat{h}\circ\gamma = \gamma'\circ\widehat{h}.$ Therefore
$\widehat{h}\circ\Gamma\subset\Gamma\circ\widehat{h}.$ Analogously,
$\Gamma\circ\widehat{h}\subset\widehat{h}\circ\Gamma.$ Thus,
$\Gamma\circ\widehat{h} = \widehat{h}\circ\Gamma.$

$(2)$ As $\kappa: \widehat{K}\to K$ is a regular  covering map, $\widehat{\Gamma}$
is a normal subgroup of $\widetilde{\Gamma}$. Suppose that for $g\in Diff(K)$ there exists
$\widehat{g}\in Diff(\widehat{K})$  lying over $g$. Consider the universal  covering map
$\widehat{\kappa}:\widetilde{K}\to \widehat{K}$. It is well known that there
is  the universal covering  map $\widetilde{\kappa}:\widetilde{K}\to K$ satisfying
the equality $\kappa\circ \widehat{\kappa}=\widetilde{\kappa}$. Hence there exists
$\widetilde{g}\in Diff(\widetilde{K})$ over $\widehat{g}$ relatively $\widehat{\kappa}$.
Note that $\widetilde{g}$ also lies over $g$ relatively $\widetilde{\kappa}$.
Therefore, according to the proved above statement $(1)$, $\widetilde{g}$ satisfies
both equalities $\widetilde{g}\circ \widehat{\Gamma} =
\widehat{\Gamma}\circ\widetilde{g}$ and $\widetilde{g} \circ
\widetilde{\Gamma} = \widetilde{\Gamma}\circ\widetilde{g} $.

Converse. Suppose that for $h\in Diff(K)$ there exists $\widetilde h\in Diff(\widetilde{K})$
satisfying the equalities: $\widehat{k}\circ\widetilde{h}=\widehat{h}\circ \widehat{k},$
$\widetilde{h}\circ\widetilde{\Gamma} =
\widetilde{\Gamma}\circ\widetilde{h}$ and $\widetilde{h}\circ
\widehat{\Gamma} = \widehat{\Gamma}\circ\widetilde{h}$. Then, according to the
assertion $(1)$, there is $\widehat{h}\in Diff(\widehat{K})$ such that
$\widehat{k}\circ\widehat{h}=\widetilde{h}\circ \widehat{k}.$ Therefore,
applying  the equality $\kappa\circ \widehat{\kappa}=
\widetilde{\kappa}$, for each $ \widehat{x}\in\widehat{K}$, we get
the chain of equalities
\begin{equation}
(\kappa\circ \widehat{h})(\widehat{x})=
\kappa\circ(\widehat{\kappa}\circ  \widetilde{h})(\widehat{\kappa}^{-1}(\widehat{x}))=
((\kappa\circ \widehat{\kappa})\circ\widetilde{h})(\widehat{\kappa}^{-1}(\widehat{x}))=
(\widetilde{\kappa}\circ \widetilde{h}) (\widehat{\kappa}^{-1}(\widehat{x}))=\nonumber\\
\end{equation}
\begin{equation}
=(h\circ \widetilde{\kappa}) (\widehat{\kappa}^{-1}(\widehat{x}))=
(h\circ(\kappa\circ \widehat{\kappa}))(\widehat{\kappa}^{-1}(\widehat{x}))
=(h\circ \kappa)(\widehat{\kappa}(\widehat{\kappa}^{-1}(\widehat{x})))
 =(h\circ \kappa)(\widehat{x}).\nonumber
\end{equation}
Hence, $\kappa\circ \widehat{h}=h\circ \kappa,$ i.e. $\widehat{h}$ lies over $h$ relatively $\kappa$.

The statement $(3)$ is obvious.

The assertion $(4)$ is a corollary of $(1)$ and $(3)$.

$(5)$ Let $G$ be the group of projections of $\widehat{G}$  and
$f:\widehat{G}\to G:\widehat{h}\mapsto h $, where $h$ is the projection
of $\widehat{h}$, is  a group epimorphism, since $f$ is surjective by
the condition. In according with the previous statement $Ker(f) = {\Gamma}$ and
$G\cong\widehat{G}/{\Gamma}.$
\end{proof}

\subsection{Proof of Theorem \ref{Th5}}
Suppose that a Cartan foliation $(M, F)$ is covered by fibration. By definition
$\ref{opr1}$ the induced  foliation  $(\widetilde{M}, \widetilde{F})$
on the space of the universal  covering $\widetilde{\kappa}:\widetilde{M}\to M$
is defined by  a locally trivial fibration $\widetilde{r}:\widetilde{M}\to B$.
Due to Theorem \ref{Th2} the regular covering map $\kappa:\widehat{M}\rightarrow M$
and locally trivial fibration $r: \widehat{M} \rightarrow B$ are defined, where $B$
is a simply connected manifold with the inducted Cartan geometry $\eta$. Let $\Psi$
be the global holonomy group of $(M, F)$, then and $\Psi$ is isomorphic to the
deck transformations group $G$ of $\kappa:\widehat{M}\rightarrow M$. Since the
manifold $\widetilde{M}$ is simply  connected, then there exists the universal
covering map $\widehat{\kappa}:\widetilde{M}\to \widehat{ M}$ satisfying the equality
$\kappa\circ \widehat{\kappa}=\widetilde{\kappa}$.  Let $\widetilde{G}$, $G$
and $\widehat{G}$  be the deck transformation groups  of the covering maps
$\widetilde{\kappa}$,  $\kappa$ and $\widehat{\kappa}$ relatively,
with $\Psi\cong G\cong\widetilde{G}/\widehat{G}$.

Let  us consider  the following preimages of the $H$-bundle ${\mathcal {R}}$ relatively
$\widetilde{\kappa}$ and  $\kappa$
\begin{equation}
\widetilde{\mathcal {R}}:=\{(\widetilde{x},u)\in \widetilde{M}\times {\mathcal {R}}\,|\,\widetilde{\kappa}(\widetilde{x})=\pi(u)\}=\widetilde{\kappa}^{*}{\mathcal {R}}\,\,\, and \nonumber
\end{equation}
\begin{equation}
\widehat{\mathcal {R}}:=\{(\widehat{x},u)\in \widehat{M}\times {\mathcal {R}}\,|
\,\kappa(\widehat{x})=\pi(u)\}=\kappa^{*}{\mathcal {R}}. \nonumber
\end{equation}

Remark that the maps
\begin{equation}
\widetilde{\theta}:\widetilde{{\mathcal {R}}}\to {\mathcal {R}}:(\widetilde{x}, u)\mapsto (\widetilde{\kappa}(\widetilde{x}), u)\,\,\,\,\, \forall(\widetilde{x}, u)\in \widetilde{{\mathcal {R}}},\nonumber
\end{equation}
\begin{equation}
\theta:\widehat{{\mathcal {R}}} \to {\mathcal {R}}:(\widehat{x}, u)\mapsto (\kappa(\widehat{x}),u) \,\,\,\,\, \forall(\widehat{x}, u)\in \widehat{{\mathcal {R}}},\nonumber\\
\end{equation}
\begin{equation}
\widehat{\theta}:\widetilde{{\mathcal {R}}}\to \widehat{{\mathcal {R}}}:(\widetilde{x}, u)\mapsto (\widehat{\kappa}(\widetilde{x}), u)\,\,\,\,\, \forall(\widetilde{x}, u)\in \widetilde{{\mathcal {R}}},\nonumber
\end{equation}
are regular covering maps  with the deck transformation groups
$\widetilde{\Gamma}$, $\Gamma$  and $\widehat{\Gamma}$, relatively,
which are isomorphic to the relevant groups $\widetilde{G}$, $G$ and
$\widehat{G}$, i.e. $\widetilde{\Gamma}\cong \widetilde{G}$,
$\Gamma\cong G$ and $\widehat{\Gamma}\cong \widehat{G}$.

Let $(\widetilde{\mathcal {R}}, \widetilde{\mathcal {F}})$ and
$(\widehat{\mathcal {R}},\widehat{\mathcal {F}})$ be the corresponding lifted
foliations. Since $(\widetilde{M}, \widetilde{F})$ and $(\widehat{M}, \widehat{F})$
are simple foliations, then $(\widetilde{\mathcal {R}},\widetilde{\mathcal {F}})$ and
$(\widehat{\mathcal {R}},\widehat{\mathcal {F}})$ are also simple foliations,
which are formed by locally trivial fibrations
$\widetilde{\pi}_{b}~:~\widetilde{\mathcal {R}}~\to~\widetilde{W}$ and
$\widehat{\pi}_{b}~:~\widehat{\mathcal {R}}~\to~\widehat{W}$. Hence
$\mathfrak{g}_{0}~(\widetilde{\mathcal {R}}, \widetilde{\mathcal {F}})~=~0$,
$\mathfrak{g}_{0}(\widehat{\mathcal {R}},\widehat{\mathcal {F}})=0$, and
$\widetilde{W}=\widetilde{\mathcal {R}}/ \widetilde{\mathcal {F}}$,
$\widehat{W}=\widehat{\mathcal {R}}/\widehat{\mathcal {F}}$ are manifolds.

Since the fibrations $\widetilde{r}:\widetilde{M}\to B$ and $r:\widehat{M}\to B$
 have the same base $B$,  each leaf of the foliation $(\widetilde{M}, \widetilde{F})$
is invariant relatively the group $\widehat{G}$, i.e.
$\widehat{G}\subset A_{L}(\widetilde{M}, \widetilde{F}).$ Therefore
$\widehat{\Gamma}\subset A_{L}(\widetilde{\mathcal {R}},\widetilde{\mathcal {F}})$
and the leaf spaces $\widetilde{\mathcal {R}}/ \widetilde{\mathcal {F}}=\widetilde{W}$
and $\widehat{\mathcal {R}}/\widehat{\mathcal {F}}= \widehat{W}$ are coincided, i.e.
$\widetilde{W}=\widehat{W}.$ Consequently basic automorphism groups
$A_{B}(\widetilde{\mathcal {R}}, \widetilde{\mathcal {F}})$ and
$A_{B}(\widehat{\mathcal {R}},\widehat{\mathcal {F}})$ may be identified.
Further we put $A_{B}(\widetilde{\mathcal {R}}, \widetilde{\mathcal {F}})=
A_{B}(\widehat{\mathcal {R}},\widehat{\mathcal {F}}).$

According to the conditions of Theorem \ref{Th5}, $\Psi$ is a
discrete subgroup of the Lie group $Aut(B,\eta)$. Let $N(\Psi)$ be the
normalizer of $\Psi$ in the Lie group $Aut(B,\eta)\cong A^H(W, \widetilde{\beta})$.
Hence, $N(\Psi)$ is a closed Lie subgroup of the Lie group $Aut(B,\eta)$
and the quotient group $N(\Psi)/ \Psi$ is also a Lie group.

Let $\pi:{\mathcal {R}}\to M$ be the projection of the foliated
bundle over $(M, F)$. Due to Theorem~\ref{Th4} the discreteness of
the global holonomy group $\Psi $ implies that the structural Lie
algebra ${\frak g}_{0}$ of the Cartan foliations $(M, F)$ is zero.
Therefore the lifted  foliation $({\mathcal {R}}, {\mathcal {F}})$ is formed by
fibres of the basic fibration $\pi_{b}:{\mathcal {R}}\to W.$

Observe that there exists a map $\tau:\widehat{W}\rightarrow W$ satisfying
the equality $\tau\circ\widehat{\pi}_{b}=\theta \circ\pi_{b}$.
It is easy to show that $\tau:\widehat{W}\to W$ is a regular covering map with
the deck transformations group  ${\Phi}$, ${\Phi}\subset A^{H}(\widehat{W},\widehat{\beta})$,
which is naturally isomorphic to the groups $\Psi$, $G$ and $\Gamma$.

Denote by $\eta=(P(B,H),\omega)$  the Cartan geometry with the
projection $p:P\to B$ onto $B$ determined in the proof of
Theorem \ref{Th2}. Remark that $\widehat{W}={P}$ is the space of the
$H$-bundle of the Cartan geometry $\eta$.

Since $\kappa:\widehat{M}\to M$ and $\pi: {\mathcal {R}}\to M$ are morphisms of
the following foliations $\kappa:(\widehat{M},\widehat{F})\to (M,F)$ and
$\theta:(\widehat{\mathcal {R}},\widehat{\mathcal {F}})\to ({\mathcal {R}},{\mathcal {F}})$ in the category
of the foliations $\frak F\frak o\frak l$, then maps $\widehat{\tau}:B\to M/F$
and $s:W\to W/H \cong M/F$ are defined, and the following diagram

\begin{equation}
\label{D}
\begin{CD}
 \xymatrix{
       P=\widehat{W}\ar@{->}[ddd]_{p}\ar@{->}[rrrr]^{\tau}&&&& W\ar@{->}[ddd]^{s}\\
       & \widetilde{\kappa}^{*}{\mathcal R}=\widetilde{\mathcal R}\ar@{->}[ul]^{\widetilde{\pi}_{B}}\ar@{->}[d]_{\widetilde{\pi}}\ar@{->}[r]_{\widehat{\theta}}&\kappa^{*}{\mathcal{R}}=\widehat{\mathcal{R}} \ar@{->}[ull]_{\widehat{\pi}_{B}}\ar@{->}[r]_{\theta}\ar@{->}[d]_{\widehat{\pi}}&{\mathcal {R}}\ar@{->}[ur]^{{\pi}_{B}}\ar@{->}[d]^{\pi}& \\
 & \widetilde{ M}\ar@{->}[ld]_{\widetilde{r}}\ar@{->}[r]^{\widehat{\kappa}} &\widehat{M}\ar@{->}[dll]_{r}\ar@{->}[r]^{\kappa}&M\ar@{->}[dr]^{q}&\\
    \widehat{M}/\widehat{F}= B\ar@{->}[rrrr]^{\widehat{\tau}}&&&& M/F  }\nonumber
\end{CD}
\end{equation}
is commutative.

Due to Proposition \ref{pr5} there are the Lie group isomorphisms
$$\varepsilon: A_{B}(M, F)\to Im(\varepsilon)\subset A^{H}(W,\widetilde{\beta})\,\, and\,\,$$
$$\widehat{\varepsilon}:A_{B}(\widetilde{M},\widetilde{F})=A_{B}(\widehat{M},
\widehat{F})\to Im(\widehat{\varepsilon})\subset A^{H}(\widehat{W}, \widehat{\beta}).$$

Let us define a map $\Theta: Im(\varepsilon)\to N(\Phi)/\Phi$ by the
following a way. Take any  $h\in Im(\varepsilon)\subset
A^{H}(W,\widetilde{\beta})$. Denote the element
$\varepsilon^{-1}(h)\in A_{B}(M, F) $ by $f\cdot A_{L}(M, F)\in
A_{B}(M, F)$, where $f\in A(M, F)$. Since $\widetilde{\kappa}:\widetilde{M}\to M$
is the universal covering map there exists $\widetilde{f}\in Diff(\widetilde{M})$ 
lying over $f$ relatively $\widetilde{\kappa}$. It not difficult to see tha
$\widetilde{f}\in A(\widetilde{M},\widetilde{F})$. Hence $\widetilde{f}\circ
A_{L}(\widetilde{M}, \widetilde{F})\in A_{B}(\widetilde{M},
\widetilde{F})$. Consider $\widehat{h}:=\widehat{\varepsilon}(\widetilde{f}\cdot
A_{L}(\widetilde{M}, \widetilde{F}))\in Im(\widehat{\varepsilon})\subset
A^{H}(\widehat{W},\widehat{\beta})$. The direct check shows that
$\widehat{h}$ lies over $h$ relatively $\tau$. Remind that $\Phi$ is
the deck transformation group  of the covering map
$\tau:\widehat{W}\to W$. Applying the statement $(1)$ of
Proposition~$\ref{pr8.2}$ we get that $\widehat{h}\in N(\Phi)$ and
the set of all automorphisms in $Im(\widehat{\varepsilon})$ lying
over $h$ is equal to the set of transformations from the class
$\widehat{h}\cdot\Phi.$ Let us put $\Theta(h):=
\widehat{h}\cdot\Phi\in N(\Phi)/\Phi.$ It is easy to check that the
map $\Theta: Im(\varepsilon)\to N(\Phi)/\Phi$ is a group
monomorphism.

The effectiveness of the Cartan geometry $\eta=(P(B,H),\omega)$ on
$B,$ where $P=\widehat{W},$ implies the existence of the Lie group
isomorphism $\sigma:A^{H}(\widehat{W},\widehat{\beta})\rightarrow
Aut(B,\eta)$ (see Remark~\ref{r4}). Observe that
$\sigma({\Phi})=\Psi$ and $\sigma( N(\widetilde{\Phi}))=N(\Psi)$,
hence there exists the inducted Lie group isomorphism
$\widetilde{\sigma}: N({\Phi})/{\Phi}\to N(\Psi)/\Psi$. Thus, the
composition of the Lie group monomorphisms
\begin{eqnarray}
\delta:=\widetilde{\sigma}\circ {\Theta}\circ\varepsilon:A_{B}(M,F)\rightarrow N(\Psi)/\Psi\nonumber
\end{eqnarray}
is the required Lie group monomorphism. Due to uniqueness  of the Lie group
structure in $A_{B}(M, F)$, in conforming with the proof of Theorem \ref{Th1},
the image $Im(\delta)$ is an open-closed subgroup of the Lie group $N({\Psi})/{\Psi}$.

\subsection{Proof of Theorem \ref{Th7}} 1. In accordance with the condition of
Theorem \ref{Th7}, $(M,F)$ is an $\frak M$-complete Cartan foliation,
and the distribution $\frak M$ is integrable. In this case, there is $q$-dimensional
foliation $(M,F^t)$ such that $TF^t=\frak M.$ Let $\widetilde{\kappa}: \widetilde{M}\to M$
be the universal covering map. As is known (\cite{Min}, Proposition 2), $\frak M$
is an integrable Ehresmann connection for the foliation  $(M,F)$.

In this case, according to the decomposition theorem belonging to
S.~Ka\-shi\-wabara \cite{Kas}, the universal covering manifold has
the form $\widetilde{M} = \widetilde{Q}\times B$, and the lifted
foliations are
$\widetilde{F}=\widetilde{\kappa}^*F=\{\widetilde{Q}\times\{y\}\,|\,
y\in B\}$, $\widetilde{F}^t=\widetilde{\kappa}^*F^t=\{\{z\}\times
B\,|\, z\in \widetilde{Q}\}$. Hence $(M,F)$ is covered by fibration
$\widetilde{r}: \widetilde{Q}\times B\to B,$ where $B$ is a simply
connected manifold. In this case, by the same way as in the proof of
Theorem \ref{Th2}, the Cartan geometry $\eta$ is induced on $B$ such
that $(M,F)$ becomes $(Aut(B,\eta),B)$-foliation.

2. Let $\Psi$ be the global holonomy group of this foliation.
Suppose now that the normalizer $N(\Psi)$ is equal to the
centralizer $Z(\Psi)$ of $\Psi$ in the group $Aut(B,\eta)$.

Let us fix points $x_0\in M$ and $(z_0,y_0)\in\widetilde{\kappa}^{-1}(x_0)\in\widetilde{M}$.
Then the fundamental group $\pi_1(M,x_0)$ acts on the universal covering space
$\widetilde{M} = \widetilde{Q}\times B$ as the deck transformation group
$\widetilde{G}\cong\pi_1(M,x_0)$ of $\widetilde{\kappa}$. Since $\widetilde{G}$
preserves both the inducted foliations $(\widetilde{M},\widetilde{F})$ and
$(\widetilde{M},\widetilde{F}^t)$, then every
$\widetilde{g}\in\widetilde{G}$ may be written in the form $\widetilde{g}=(\psi^t,\psi)$,
where $\psi^t$ forms a subgroup $\Psi^t$ in $Diff(\widetilde{Q})$,
$\psi\in\Psi$, and $\widetilde{g}(z,y)=(\psi^t(z),\psi(y))$,
$(z,y)\in\widetilde{Q}\times B.$  The maps $\widetilde{\chi}:
\widetilde{G}\to\Psi: \widetilde{g}\to\psi$ and $\widetilde{\chi}^t: \widetilde{G}\to\Psi^t:
\widetilde{g}\to\psi^t$ are the group epimorphisms.

Let $h$ be any element from
$N(\Psi)/\Psi$. Since $N(\Psi) = Z(\Psi)$, we have the following chain of equalities
$$\widetilde{g}\circ (id_{\widetilde{Q}}, h) =({\psi}^t,\psi)\circ
(id_{\widetilde{Q}}, h)=({\psi}^{t} \circ id_{\widetilde{Q}}, \psi \circ h)=\\$$
$$(id_{\widetilde{Q}} \circ {\psi}^{t}, h\circ \psi )= (id_{\widetilde{Q}}, h)\circ
({\psi}^t,\psi) = (id_{\widetilde{Q}}, h)\circ \widetilde{g}$$
for any $\widetilde{g}= ({\psi}^t,\psi)\in\widetilde{G}$, i.e.
$\widetilde{G}\circ(id_{\widetilde{Q}},h)=(id_{\widetilde{Q}},h)\circ\widetilde{G}$.
Therefore, by the state\-ment~$(1)$ of Proposition \ref{pr8.2} for the deck transformation
group $\widetilde{G}$, there exists $\widetilde{h}\in Diff(M)$
such that $(id_{\widetilde{Q}},h)$ lies over $\widetilde h$ relatively to
$\widetilde{\kappa}:\widetilde{M}\to M$.

Using $(id_{\widehat{Q}},h)\in A(\widetilde{M},\widetilde{F})$ it is not difficult
to check that $\widetilde{h}\in A(M,F)$. Hence, $\varepsilon(\widetilde{h}\cdot A_{L}(M,F) )= h.$
This means that $\varepsilon: A_B(M,F)\to N(\Psi)/\Psi$ is surjective.
Thus, $\varepsilon$ is the group isomorphism.
\section{Examples of the calculation of the basic\\ auto\-mor\-phisms groups}\label{ss_5.11}
\begin{defn} Let $\xi=(P(N,H),\omega)$ is arbitrary Cartan geometry of the type
$(G,H)$, of the effectiveness of which is not assumed. The group
$$\mathrm{Gauge}(\xi):=\{\Gamma\in{ A}(\xi)\mid p\circ\Gamma=p\}$$
is called {\it of the gauge transformation group of the Cartan geometry} $\xi.$
\end{defn}

\begin{ex}
 \label{E1}
\label{Ex1} Let $G$ be a Lie group and  $H$ be a closed subgroup of $G$.
Denote  by  $\mathfrak{g}$ and $\mathfrak{h}$
 the Lie algebras of Lie  groups $G$ and $H$ relatively. Suppose that the
kernel of the  pair of Lie groups  $(G,H)$ is equal to the  intersection
$K=Z(G)\cap Z(H)$ of the centers of the groups
 $G$ and $H$. Denote by $\omega_{G}$ the Maurer-Cartan $\mathfrak{g}$-valued
$1$-form on the Lie group $G.$ Then $\xi^{0}=(G(G/H,H),\omega_{G})$
is the Cartan geometry, and its transverse curvature is zero. Consider any
smooth manifold $L$ .  Denote by $M$ the product of
manifolds $M=L\times (G/H),$ and ${F}=\{L\times\{x\}\,|\, x\in G/H\}.$
Then $(M,{ F})$ is the trivial transverse  homogeneous foliation
with the transverse Cartan geometry $\xi^0.$
Because the foliation $(M,{ F})$ is trivial, the group  ${ A}_B(M,{ F})$ is  coincided
with the  automorphisms group  $Aut(\xi^{0})$ of the Cartan
geometry $\xi^0$ in the category $\frak C \frak a \frak r$.

Any left action $L_g$, $g\in G$, of the Lie group $G$ satisfies the
conditions: $L_g^*\omega_{G}=\omega_{G}$ and $L_g\circ R_a=R_a\circ
L_g\,$ $\forall a\in G.$ Therefore, $L_g\in Aut(\xi^{0})$ and
$\dim(Aut(\xi^{0})) = \dim(G)= \dim(\frak g).$ By assumption,
the kernel of the pair $(G,H)$ equal to  $K=Z(G)\cap Z(H)$,  hence
$Gauge(\xi^0)=\{L_b\,|\, b\in~K\}$. Thus, the basic automorphisms
group  ${ A}_B(M,{ F})$ is equal  to the quotient $Aut(\xi^{0})/Gauge(\xi^0) \cong G/K$,
and $\dim({ A}_B(M,{ F})) = \dim(\frak g) - \dim(\frak k),$
where $\frak k$ is the algebra Lie of the kernel  $K$.
\end{ex}

Example \ref{E1} shows  that the estimation \eqref{oz1}
of  the dimension group ${ Aut}_B(M, {F})$ in
Theorem~\ref{Th1} is exact.

\begin{ex}\label{E2}
Let $$\left\lbrace
G:=\begin{pmatrix}
1 & x & y\\
0 & 1 & x\\
0 & 0 & 1\\
\end{pmatrix}
\,|\,x,y\in\mathbbm{ R}^1
\right\rbrace, \,\,\, 
K:=\left\lbrace \begin{pmatrix}
1 & 0 & y\\
0 & 1 & 0\\
0 & 0 & 1\\
\end{pmatrix}
\,|\, y \in\mathbbm{ R}^1
\right\rbrace.
 $$

Then $G$ is an Abelian Lie group and  $K$ is a connected closed subgroup of the
group Lie $G$. Hence $K=Z(G)\cap~Z(K)$. Let
$\xi^0=(G(G/K,K),\omega_G)$ is canonical Cartan geometry with
projection $p: G\to G/K.$

The  foliation $(G, {F})$, where $F=\{gK\,|\, g\in G\}$, is a proper
Cartan foliation with transverse Cartan geometry $\xi^0$. Using
Proposition \ref{pr5.2} we see that $\mathcal{R}=G$ and $H=\{e\}$.
Therefore the basic automorphisms group $A_{B}(M, F)$ is isomorphic
to  the Lie  group $Aut(G/K,\xi^{0})\cong G/K$. Thus $A_{B}(M,
F)\cong G/K.$

Since $K=H$, then
$\dim (A_{B}(M, F))=\dim (G) -\dim {(K)}=\dim {\mathfrak{g}}-\dim {\mathfrak{k}}$
and the estimation \eqref{oz1} of Theorem \ref{Th1} is exact.
\end{ex}

\begin{ex}\label{E3} Let $G$ be the similar group of the Euclidean space
$\mathbb{E}^q,$ $q\ge 1$. Then $G =
CO(q)\ltimes\mathbb{R}^q$ is the semidirect product of
the conformal group $CO(q)$ and the group
$\mathbb{R}^q.$ Let $H = CO(q)$ and $p\colon G\to
G/H=\mathbb{E}^q$ be the canonical principal $H$-bundle. Let
$\frak g$ be the Lie algebra of the Lie group $G,$ and $\omega_{G}$ be
the Maurer-Cartan $\frak g$-valued $1$-form on $G.$ Then
$\xi=(G(\mathbb{E}^q,H),\omega_{G})$ is an effective Cartan geometry.
Foliations with this transverse geometry $(\mathbb{E}^q,\xi)$ are
called {\it transversally similar foliations}~\cite{Min}.

Let $Q$ be a smooth $p$-dimensional manifold whose fundamental group
$\pi_1(Q,x)$ contains an element $\alpha$ of infinite order. For an
arbitrary natural number $q\ge 1,$ denote by $\mathbb{E}^q$ the
$q$-dimensional Euclidean space. 

Define a group homomorphism
$\rho\colon\pi_1(Q,x)\to{ Diff}(\mathbb{E}^q)$ by setting
$\rho(\alpha)=\psi,$ where $\psi$ is the homothety transformation of
the Euclidean space $\mathbb{E}^q$ with the coefficient $\lambda\neq
1,$ i.~e. $\psi(x)=\lambda x\,$ $\forall x\in\mathbb{E}^q,$ and
$\rho(\beta)=\mathrm{id}_{\mathbb{E}^q}$ for any element
$\beta\in\pi_1(Q,x)$ such that $\beta\neq \alpha^k$ with every
integer $k.$ Then $(M,F)=\mathrm{Sus}({\mathbb E}^q,Q,\rho)$ is a
proper transversally similar foliation with a unique closed leaf
diffeomorphic to the manifold $Q.$

Due to $N(\Psi)=Z(\Psi)$, according to Theorem \ref{Th7}, we get
${ A}_B(M,F)\cong N(\Psi)/\Psi.$ The foliation  $(M, F)$
is covered by the fibration $\widetilde{Q}\times \mathbb{E}^{q}\to \mathbb{E}^{q}$
where $\widetilde{Q}\to Q$ is the universal covering map.
Hence $\Psi:=\rho(\pi_{1}(Q, x))\cong \mathbb{Z}$ is the global holonomy group
of $(M, F)$ and $K=H$ is the kernel of the pair $(G,H)$. Thus the assumption
of which was made in Example $\ref{E1}$ is realized.

In our case $\Psi=\langle\psi\rangle$ and
$N(\Psi)=CO(q)=\mathbb{R}^+\cdot O(q),$ therefore ${A}_B(M,F)\cong
U(1)\times O(q),$ where $U(1)\cong\mathbb{R}^+/\Psi$ is the compact
$1$-dimensional Abelian group.

If $q=1,$ then $O(q)=\mathbb{Z}_2$ and ${ A}_B(M,F)\cong
U(1)\times \mathbb{Z}_2.$

Thus $\dim (A_{B}(M, F))=\dim {\mathfrak{h}}$ and the estimate~\eqref{oz2}
in Theorem $\ref{Th1}$ is exact.
\end{ex}
\begin{ex}\label{E4} Let $\mathbb{E}^2 = (\mathbb{R}^2,g)$ be
an Euclidean plane with an Euclidean metric $g$.
Let $\psi$ be the rotation of the Euclidean plane $\mathbb{E}^2$
around the point $0\in\mathbb{E}^2$ by the angle
$\delta=2\pi r.$ Denote by $\frak{I}(\mathbb{E}^2)$ the full isometry
group of $\mathbb{E}^2.$ It is well known that
$\frak{I}(\mathbb{E}^2) \cong O(2)\ltimes\mathbb{R}^2$.

Let $\rho\colon\pi_1(S^1,b)\cong\mathbb{Z}\to\frak{I}(\mathbb{E}^2)$ be
defined by the equality $\rho(1):=\psi,$ $1\in\mathbb{Z}.$ Then we
have the suspended Riemannian foliation
$(M, F):= Sus(\mathbb{E}^2,S^1,\rho).$ This foliation has a
unique closed leaf which is compact.

Due to $N(\Psi)=Z(\Psi)=O(2)$ Theorem \ref{Th7} is applicable.
Consequently, ${A}_B(M,F)\cong N(\Psi)/\Psi=O(2)/\Psi$.  Hence
${A}_B(M,F)$ admits a Lie group structure if and only if $\Psi$ is a
closed subgroup of $O(2)$ or, equivalent, when $\delta=2\pi r$ for
some rational number $r.$

If $\delta=2\pi r,$ where $r$ is a nonzero rational number, then
$A_B(M,F)\cong O(2).$
\end{ex}

\begin{ex}\label{E5}
Consider the standard $2$-dimension torus $\mathbb{T}^{2}=\mathbb{R}^{2}/\mathbb{Z}^{2}$
and call the pair of vectors  $e_{1}=\left( \begin{array}{ccc}
1\\
0\\
\end{array}\right),$ $e_{2}=\left(\begin{array}{ccc}
0\\
1\\
\end{array}\right)$ by the standard basis of the tangent vector space
$T_{x}\mathbb{T}^{2}$ with $x\in \mathbb{T}^{2}$.
Let $\Omega:\mathbb{R}^{2}\to \mathbb{T}^{2}$ be the quotient map,
which is the universal covering of the torus. Denote by $f_{A}$ the Anosov
automorphism of the torus $\mathbb{T}^{2}$ determined by the matrix
$A\in SL(2, \mathbb{Z})$, while by $E$ the identity $2\times 2$ matrix.

Let $g$~be the flat Lorentzian metric on the torus $\mathbb{T}^2$
given in the standard basis by the matrix $\lambda\left(
\begin{array}{ccc}
2&m\\
m&2\\
\end{array}\right),$ where $\lambda$~ is any non zero real number and
$m\in\mathbb Z$, $|m|>2.$
Introduce notations $\frak{I}(\mathbb{T}^2,g)$ for the full isometry group  of this
Lorentzian torus $(\mathbb{T}^2,g)$ and ${\frak I}_0(\mathbb{T}^2,g)$ for
the stationary subgroup of the group ${\frak I}(\mathbb{T}^2,g)$ at point
$0 = \Omega (0)$, $0 = (0,0)\in\mathbb R^2$. As is known (\cite{ZR}, Example 3),
$\frak{I}(\mathbb{T}^2,g) = \frak{I}_0(\mathbb{T}^2,g)\ltimes \mathbb Z \times\mathbb Z,$
where the group $\Phi_0: = \frak{I}_0(\mathbb{T}^2,g)$ is generated by
$f_A$, $f_{\widetilde{A}}$ and $-E$,  $A=\left(
\begin{array}{ccc}
m&1\\
-1&0
\end{array}
\right)$ and
$\widetilde{A}=
\left(
\begin{array}{ccc}
0&1\\
1&0\\
\end{array}
\right)$, hence ${\frak
I}(\mathbb{T}^2,g)\cong(\mathbb{Z}_2\times\mathbb{Z}_2\times\mathbb{Z})\ltimes T^2$.

Let $Q=S^{1}$  and $T =\mathbb T^2$.
 Define the group homomorphism
$\rho: \pi_1(S^{1})\cong \mathbb{Z} \to {\frak I}(\mathbb{T}^2,g)$ by the equality
$\rho(k): = (f_A)^{k}$, $k\in\mathbb{Z}$.
Then the suspended foliation $(M, F): = Sus(\mathbb{T}^{2},S^{1},\rho)$ is
Lorentzian, and its global holonomy group $\Psi$ is the group of all
transformations lying over the group $\Phi: = <f_A>$ relatively the
universal covering map $\Omega: \mathbb{R}^2\to \mathbb{T}^2$.

Elements of affine group $Aff(A^2)$ will be denoted by $<C,c>$, $C\in GL(2,\mathbb{R})$,
$c\in R^2$, in compliance with $Aff(A^2) = GL(2,\mathbb{R})\ltimes {R}^2$.
The composition of transformations
$<C,c>$ and $<D,d>$ from $Aff(A^2)$ has the following form
$<C,c> <D,d> = <C D, C d + c>$.

The check using the Proposi\-ti\-on~\ref{pr8.2} shows that
$\Psi = \Psi_0^0\ltimes (\mathbb Z\times\mathbb Z)$, where
the group $\Psi_0^0$ is generated by matrix $A$, i.e. $\Psi_0^0 \cong\Phi$. Let
$\Gamma: = \mathbb Z\times\mathbb Z \subset \frak{I}(\mathbb{E}^2)$.

Consider any $<C,c>\in N(\Psi)$, then
for every $<E,a>\in \Gamma$ there are\\ $<D,d>, <K,b>\in\Psi$ such that
\begin{equation} \label{q1}
<C,c> <E,a>=<D,d> <C,c>,
\end{equation}
\begin{equation} \label{q2}
<C,c> <K,d>=<E,a> <C,c>.
\end{equation}
Hence $D = E$, $K = E$ and $<C,c> \in N(\Gamma)$.
Consequently $N(\Psi) \subset N(\Gamma)$ and,
due to $\Gamma$ is the deck transformation group of $\Omega$,
by the statement $(2)$ of Proposition \ref{pr8.2}, the following map
\begin{equation}
\alpha:  N(\Psi)\to\frak{I}(\mathbb T^2,g): \widehat{h}\mapsto h,\nonumber
\end{equation}
where $\widehat{h}$ lies over $h$ relatively $\Omega: \mathbb{R}^2\to \mathbb {T}^2$,
is defined and it is a group homomorphism.

The relations \eqref{q1} and \eqref{q2} imply also that $<C,0>\in
N(\Gamma)$. Therefore $f_C\in \Phi_0: = \frak I_0(\mathbb {T}^2,g)$
and $C\in\Psi_0,$ where $\Psi_0$ is the subgroup of $N(\Psi)$
generated by matrixes $A$, $\widetilde A$ and $-E$, i.e. $\Psi_0
\cong\Phi_0$. Thus, the stationary subgroup $N(\Psi)_0$ at $0\in
\mathbb{R}^2$ of the normalizer $N(\Psi)$ is equal to $\Psi_0$.

Since $\alpha(\Psi) = \Phi$, then the group homomorphism $\alpha$ has the property
$\alpha(N(\Psi)) = N(\Phi)$.

Let us compute the normalizer $N(\Phi)$ of $\Phi$ in the group $\frak{I}(\mathbb{T}^2,g)$.
Take any $<D,d>$ from $N(\Phi)$. Then there is $k\in \mathbb Z \setminus \{0\}$
such that
\begin{equation}
<D,d> <A,0> = <A^k,0> <D,d>,\nonumber
\end{equation}
consequently $A^k d = d$, i.e. $1$ is the eigenvalue of $A^k$.
Since $A^k$ is an Anosov automorphism, then its eigenvalues
are irrational. Thus, it is necessary $d = 0,$
hence $N(\Phi)\subset\Phi_0$.  Observe that $A \widetilde{A} = \widetilde{A} A^{-1}$
and $A (-E) =  (-E) A$, therefore, $N(\Phi) = \Phi_0$.
Thus, $N(\Phi)/\Phi = \Phi_0/\Phi \cong\mathbb{Z}_2\times\mathbb{Z}_2$,
i.e. $N(\Phi)/\Phi\cong\mathbb{Z}_2\times\mathbb{Z}_2$.

By Theorem \ref{Th5} there is the group mono\-morphism
$\varepsilon: A_{B}(M, F)\to N(\Psi)/\Psi$ and the image $Im (\varepsilon)$ is
an open-closed subgroup  of $N(\Psi)/\Psi$. Show that $\varepsilon$ is a surjection.

Note that $N(\Psi)=\Psi_{0}\ltimes \mathbb{Z}\times\mathbb{Z}$ and the quotient group
$N(\Psi)/\Psi$ is generated by transformations $<\widetilde{A}, 0>$
and $<-E, 0>$ of $\mathbb{R}^2$. Every transformations of product
 $\mathbb{R}^{1}\times\mathbb{R}^2$ conserving this product may be written
as a pair $(g_{1}, g_{2})$, where $g_{1}\in Diff(\mathbb{R}^{1})$,
$g_{2}\in Diff(\mathbb{R}^{2})$, and $(g_{1}, g_{2})(t, z)=
(g_{1}(t),g_{2}(z))\,,\, (t,z)\in\mathbb{R}^{1}\times\mathbb{R}^{2}$.
 Let $\widetilde{G}=\{(\widetilde{g}_{1}, \widetilde{g}_{2})\}$
 be the group of covering  transformations  of the universal covering map
$\widetilde{\kappa}:\mathbb{R}^{1}\times\mathbb{R}^{2}\to M$, then
$\widetilde{g}_{1}:\mathbb{R}^{1}\to\mathbb{R}^{1}:t\mapsto t+m, \, m\in \mathbb{Z},\, \widetilde{g}_{2}\in \Psi $.
 Let $\widetilde{f}: \mathbb{R}^{1}\to\mathbb{R}^{1}: t\mapsto -t$
be the diffeomorphism of $\mathbb{R}^{1}$. Observe that
$(\widetilde{f},\widetilde{A})\in A(\mathbb{R}^{3}, \widetilde{F}_{tr}),$
where $\widetilde{F}_{tr}=\{\mathbb{R}^{1}\times\{z\}\,|\,z\in \mathbb{R}^{2}\}$,
and $(\widetilde{f},\widetilde{A})$ lies over $(\widetilde{A}, 0)$ of $\mathbb{R}^{2}$
relatively the canonical projection onto second multiplier
$\widetilde{r}:\mathbb{R}^{1}\times\mathbb{R}^{2}\to\mathbb{R}^{2}$.
Since $(f, A)\circ (\widetilde{f},\widetilde{A})=(\widetilde{f},\widetilde{A})\circ (f^{-1},A^{-1})$,
we see that $(\widetilde{f},\widetilde{A})\circ \widetilde{G}=\widetilde{G}\circ (\widetilde{f},\widetilde{A})$.
Hence, there exists $\widetilde{h}\in A(M, F) $ such that $(\widetilde{f},\widetilde{A})$
lies over $\widetilde{h}$ relative $\widetilde{\kappa}$. It means that
$\varepsilon(\widetilde{h}\cdot A_{L}(M, F))=<\widetilde{A}, 0>$.
The existence $h'\in A(M, F)$ such that $\varepsilon({h'}\cdot A_{L}(M, F))=<-E, 0>$
is cheeked by the same way as in the proof of Theorem \ref{Th7}. Therefore
the group homomorphism $\varepsilon: {A_{B}(M, F)\to N(\Psi)/\Psi}$
is surjection, hence $\varepsilon$ is the group isomorphism.

Due to the following  chain  of group isomorphisms $A_B(M,F)\cong N(\Psi)/\Psi\\
\cong (N(\Psi)/\Gamma)/(\Psi/\Gamma) \cong N(\Phi)/\Phi \cong\mathbb{Z}_2\times\mathbb{Z}_2$,
we have
\begin{equation}
A_B(M,F)\cong\mathbb{Z}_2\times\mathbb{Z}_2.\nonumber
\end{equation}
\end{ex}

\begin{rem} It is well known (see, for example, \cite{Nit}, Lemma 3.3) that
the set of periodic orbits of a Anosov automorphism of the torus
$\mathbb{T}^2$ is countable. Therefore the foliation $(M, F)$
constructed in Example~\ref{E5} has a countable set of closed leaves
and according to the item $(b)$ of Theorem \ref{Th1} its basic automorphism group
$A_B(M,F)$ is a discrete Lie group. Our result
$A_B(M,F)\cong\mathbb{Z}_2\times\mathbb{Z}_2\nonumber$ illustrates
this assertion.
\end{rem}

\end{document}